\documentclass[11pt,a4paper,oneside,reqno]{amsart}

\newtheorem{theorem}{Theorem}[section]
\newtheorem{lemma}[theorem]{Lemma}
\usepackage{setspace}
\theoremstyle{definition}
\newtheorem{definition}[theorem]{Definition}
\newtheorem{example}[theorem]{Example}
\newtheorem{proposition}[theorem]{Proposition}
\newtheorem{corollary}[theorem]{Corollary}

\newtheorem{fact}[theorem]{Fact}
\newtheorem{remark}[theorem]{Remark}

\newtheorem{remark and question}[theorem]{Remark and Question}
\newtheorem{Notation and Remark}[theorem]{Notation and Remark}
\newtheorem{Notation}[theorem]{Notation}
\newtheorem{question}[theorem]{Question}
\usepackage{amsmath}
\usepackage{enumitem}
\usepackage{amssymb}
\usepackage{lipsum}
\usepackage{latexsym}
\usepackage{fullpage}
\usepackage{comment}
\usepackage{xspace} 
\DeclareMathOperator{\alfa}{\beta}

\newcommand{\N}[2][k]{N_{#2}(<#1)}

\newcommand{\Ralfak}[1][R]{{#1}[\alfa_1,\ldots,\alfa_k]}

\newcommand{\Null}[1][R]{N_{#1}}
\newcommand{\Nulld}[1][R]{N'_{#1}}
\DeclareMathOperator{\im}{Im}

\newcommand{\Stabk}[1][R]{St_{k}(#1)}
\newcommand{\PrPol}[1][R]{\mathcal{P}(#1)}
\newcommand{\PolFun}[1][R]{\mathcal{F}(#1)}

\DeclareMathOperator{\quv}{\xspace{ }\triangleq\xspace{ }}

\usepackage{mathtools}
\usepackage{epsf}
\usepackage{layout}
\usepackage{hyperref}      
\begin{document}
	\openup .3em
	
	\title[]{Polynomial functions on a class of finite non-commutative  rings}
	
	\author{Amr Ali Abdulkader Al-Maktry}
	\address{\hspace{-12pt}Department of Analysis and Number Theory (5010) \\
		Technische Universit\"at Graz \\
		Kopernikusgasse 24/II \\
		8010 Graz, Austria}
	\curraddr{}
	\email{almaktry@math.tugraz.at}
	\author{Susan F.  El-Deken}
	\address{\hspace{-12pt}Department of Mathematics \\
		Faculty of Science \\
		 Capital University
		Ain Helwan, 11790 \\
		Cairo, Egypt}
	
	\email{sfdeken@science.helwan.edu.eg}
	\curraddr{}
	\email{}

	
	
	\date{}
	
	\dedicatory{}
	\maketitle

	\section*{Abstract}
Let $R$ be a finite non-commutative ring with $1\ne 0$. By a polynomial function on  $R$, we mean a  function $F\colon R\longrightarrow R$ induced by a polynomial $f=\sum\limits_{i=0}^{n}a_ix^i\in R[x]$ via right substitution of the variable $x$, i.e., 
$F(b)= \sum\limits_{i=0}^{n}a_ib^i$ for every $b\in R$. In this paper, we  study the   polynomial functions of the free $R$-algebra $R[\beta_1,\ldots,\beta_k]$, for $k\ge 1$,
with a   basis $\{1,\beta_1,\ldots,\beta_k\}$ consisting of central elements  satisfying  $\beta_i\beta_j=0$ for every $1\le i,j\le k$. 
Our investigation revolves around assigning a polynomial $\lambda_f(y,z)$ over $R$ in non-commuting variables  $y$ and $z$ to each polynomial $f$ in $R[x]$; and describing the polynomial functions on $R[\beta_1,\ldots,\beta_k]$ through the polynomial functions induced on $R$  by polynomials in $R[x]$ and by their assigned polynomials in the non-commuting variables  $y$ and $z$. 
By extending results from the commutative case to the non-commutative scenario, we demonstrate that several properties and theorems in the commutative case can be generalized to the non-commutative setting with appropriate adjustments.

			\section{Introduction}\label{CHSIntro} 
		
		Let $R$ denote a finite   ring with identity $1\ne 0$. By a polynomial function on  $R$, we mean a right-polynomial function, which is   a function $F\colon R\longrightarrow R$ induced by a polynomial $f=\sum\limits_{i=0}^{n}a_ix^i\in R[x]$ via right substitution of the variable $x$, i.e., $F(b)= \sum\limits_{i=0}^{n}a_ib^i$ for every $b\in R$. In particular,   if $F$ is a bijection,  we call   $F$   a \emph{polynomial permutation} and $f$   a \emph{permutation polynomial} 
		on $R$.   We denote the set of polynomial functions on $R$ by $\PolFun[R]$, and  by $\PrPol[R]$ we denote the set of all polynomial permutations on $R$. When $R$ is a finite commutative ring,  the set $\PolFun[R]$ is known to be a monoid  under the composition of functions. Further, the group of units of $\PolFun[R]$ is just the set  $\PrPol[R]$. While the case that $R$ is a non-commutative ring  is  difficult and still open,  we can, however,  ensure that the set $\PolFun[R]$ and the set $\PrPol[R]$ generate a monoid and a group, respectively, of functions on $R$.
        
		Even though polynomial functions on finite fields have been investigated before the 20th  century, polynomial functions on finite rings that   are not   fields were first studied at the beginning of the third decade of the 20th century by    
		Kempner~\cite{Residue}. 
		Kempner intensely studied the polynomial functions of the ring $\mathbb{Z}_m$ of integers modulo $m$.  Perhaps the only defect of his work was a lack of simplicity, but his work nevertheless had a significant impact on other researchers for decades.  For instance,   some researchers followed up on his results, derived equivalent relations, and added new ideas to the topic as well \cite{per1,pol1,pol2,pol3}.
		Others generalized his results to Galois rings ~\cite{gal}, finite commutative principal rings~\cite{Necha}, and the so-called suitable rings~\cite{suit}. In most of the previously mentioned works, the authors tackled the problems of counting  the number of elements in the monoid $\PolFun[R]$  and
		in the group   $\PrPol[R]$, and finding canonical representations of the polynomial functions.
		
		Recently, authors have provided algorithms determining whether or not a given function is a polynomial function
		and if so,  it obtains its polynomial representative on the ring of integers modulo $m$ \cite{guhaalgo} and on any finite commutative rings~\cite{gabalgo}.
		
		To  investigate polynomial functions efficiently on finite rings,  one needs some essential knowledge of those polynomials that induce the zero function on the ring $R$. We refer to these polynomials as null polynomials on $R$.
		These polynomials are of great interest in this investigation. Because, for example, one can represent all polynomial functions
		by polynomials of degrees not exceeding $n-1$ whenever a monic null polynomial of degree $n$ is given.  Another reason is  the fact that null polynomials form an ideal of the polynomial ring over  a quite large class of finite rings, including commutative rings, rings of matrices over commutative rings, and local rings
		(see Section~\ref{sc2}). Because of this,  some papers mainly considered  null polynomials  (see for example \cite{Gilmnull,Banda,RogChain}). 
		
		Beyond the significant impact of Kempner's work in the theory, the work of N\"obauer on polynomial permutations and permutation polynomials cannot be underestimated. We cannot here mention his enormous results in the theory, and rather we give a limited brief description of his work and refer the interested reader to his book with Lausch~\cite{Nobook}. For instance,  he obtained a condition determining whether a given polynomial induces a permutation of the ring $\mathbb{Z}_{p^n}$ of integers modulo a prime power   (\cite[Hilfssatz 8]{perconmod}), and later he generalized this condition to any finite commutative  rings~\cite{percondgen}. Also, he expressed the group $\PrPol[\mathbb{Z}_{p^n}]$ as the wreath product
		of   two of its subgroups~\cite{per1}. This structural result was first generalized to finite commutative local rings
		of nilpotency index $2$ by Frisch~\cite{suit} and finally into any finite local ring by~G\"orcs\"os, Horv\'ath and M\'esz\'aros~\cite{Gabor}. Further, some other papers considered the structures of subgroups of the group $\PrPol$ for the case $R=\mathbb{Z}_{p^n}$~\cite{per2}  or  for any finite commutative ring  
		\cite{hakiunit}.  
		
		Despite tackling various aspects related to polynomial functions on finite rings,
		all rings examined in the previously mentioned references are commutative.  Unlike the commutative case, polynomial functions on finite non-commutative rings only arose in special circumstances to answer specified questions.
		Indeed, most of the theory considered only scalar polynomial functions on the ring of $n\times n$ matrices over a finite field, $M_n(\mathbb{F}_q)$, (see, for example, \cite{Brawscal,Brawelyinvert}),  or over a finite commutative local ring~\cite{matringpolybr} for particular considerations.
		For instance,  they considered scalar   polynomial functions
		that permute  the entire ring of matrices~\cite{matringpolybr}; or that permute  a subset  of $M_n(\mathbb{F}_q)$ \cite{polsym}.
		Additionally,   in \cite{Brawscal}  and \cite {Brawcount}, the number of all scalar polynomial functions and the  number of all scalar polynomial permutations, respectively, on the ring $M_n(\mathbb{F}_q)$ have been counted. Meanwhile, in  \cite{Wernerki,frischpolytri}, the authors focused on null polynomials and their connection with the ring of integer-valued polynomials of certain non-commutative algebras. 
		
		To the best of our knowledge, the only references containing counting formulas for the number of polynomial functions 
		on finite non-commutative rings are \cite{chainpch} and \cite{BrawleyRight}. In \cite{chainpch}, the author not only found the number of polynomial functions but also the number of polynomial permutations of any finite non-commutative chain rings (Galois-Eisenstein-Ore ring)  of characteristic $p$. While in \cite{BrawleyRight}, the author counted the polynomial functions of the ring of matrices over finite fields. 
				
		In this paper, we delve into the realm of polynomial functions within non-commutative rings.  We aim  to explore the polynomial functions on the non-commutative free $R$-algebra $R_k=R[\alfa_1,\ldots,\alfa_k]$, where $\alfa_i$ lies in the center $C(R_k)$ of $R _k$ and  $\alfa_i\alfa_j=0$ for $i,j=1,\ldots,k$. Since this ring can be expressed as the quotient ring
		$R[x_1,\ldots,x_k]/T$, where $T$ is the ideal generated by the set $\{x_ix_j\mid i,j=1,\ldots,k\}$, we call it the ring of dual numbers in $k$ variables. When $R$ is commutative, exploring $\PolFun[R_k]$ has been done in two stages. Initially, in \cite{Haki}, only the case $k=1$  was examined with more consideration on the ring of integers modulo $p^n$. Then in \cite{polysev},   $\PolFun[R_k]$ was examined for any positive integer $k$ and any finite commutative ring $R$. In general, these examinations were carried out not only by examining the polynomial functions obtained by polynomials in $R[x]$ but also by examining the polynomial functions obtained by their formal derivatives. 
		
		We will show that most of the results of   \cite{polysev} are still true, and losing an important property like commutativity does not affect the results. Furthermore,  given a polynomial $f\in R[x]$, we assign a new polynomial $\lambda_{ f}(y,z) $ in   non-commuting variables $y$ and $z$ 
		(see Definition~\ref{assignpol} for details); and  similar to the commutative case, investigating   $\PolFun[R_k]$
		will depend not on the  polynomials from $R[x]$ but also on their assigned polynomials in the non-commuting variables $y$ and $z$. In other words,  conditions derived on $f$ and its formal derivative $f'$ in \cite{polysev}  are  revisited here by  general conditions on $f$ and $\lambda_{ f}(y,z) $
		(compare, for example, Theorem~\ref{CHS4} with \cite[Theorem~3.4]{polysev},  and Theorem~\ref{Cherper} with  \cite[Theorem~4.1]{polysev}). Furthermore, we also consider the case where $R$ is a chain ring with  $Char(R)\ne p $  in some detail.
		
		The influences of polynomial functions can be seen in different topics,  for example dynamical systems  \cite{Anash},
		computer science (see for example \cite{Com1,Com2}) and in discrete mathematics~\cite{discapp}. Further,  beyond these applications, polynomial functions appear in different
		contexts of mathematics, for instance, permutation polynomials occur as  $R$-automorphisms of the polynomial
		ring (see for example \cite{Hajja,Hajja2}) and as isomorphisms of combinatorial objects~\cite{Cobj1}. Also,  some groups of polynomial permutations occur as subgroups of the permutation groups of cyclic codes~\cite{kensa}. Moreover, null polynomials are efficiently utilized in investigating several rings of integer-valued polynomials (see for example~\cite{NWer,Javad}).
		In our opinion, therefore, exploring polynomial functions on new classes of non-commutative rings opens new gates for 	applications.
		
		It is worth mentioning that   some mathematicians have explored    polynomial functions on other algebraic structures such as semigroups~\cite{semi},  monoids~\cite{Tichyfinm}, groups~\cite{polgr}, algebras~\cite{PolAlg}  and lattices~\cite{Eiglat}.

		Here is a brief description of the paper. In Section~\ref{sc2}, we review basic definitions and facts about polynomial functions on finite non-commutative rings.   Section~\ref{sec3}  contains the essential properties of the ring of dual numbers in $k$ variables, $R_k$, and their polynomials. In Section~\ref{CHSsc3}, we describe null polynomials on $R_k$. 
		In Section~\ref{Sec04}, we characterize permutation polynomials on $R_k$. Then in its subsection, 
		we examine permutation polynomials on $R_k$ for the case $R$ is a chain ring.		
		Finally,  we devote Section~\ref{6}  to investigate some  groups generated by permutation polynomials on $R_k$    with more consideration on those generated by permutations  represented by polynomials over $R$.

		\section{Basic facts on polynomial functions }\label{sc2}
		In this section, we expose some basic  definitions and facts  about polynomial functions on finite  non-commutative rings. 
		
		Throughout this paper,  by $g'$, we denote the formal derivative of the polynomial g;  and let $k$ denote a positive integer and let $\mathbb{F}_q$ denote the finite field of $q$ elements.
		
		\begin{definition}\label{CHSequvfun}
			Let $R$ be a non-commutative ring, and  $g=\sum\limits_{i=0}^{n}a_ix^i\in R[x]$. Then:
			\begin{enumerate}
				\item The polynomial $g$ induces two functions $F_1,F_2\colon R\longrightarrow R$ by right
				substitution $F_1(b) =g_r(b)=\sum\limits_{i=0}^{n}a_ib^i $  and left substitution $F_2(b) =g_l(b)=\sum\limits_{i=0}^{n}b^ia_i $ for the variable $x$. We call $F_1$ a right-polynomial function
				on $R$	and  $F_2$ a left-polynomial function on $R$.  
				
				\item By $[g_r]_R$,	we denote the right-polynomial function induced by $g$ on $R$, and  
				by $[g_l]_R$,	we denote the left-polynomial function induced by $g$ on $R$. When the ring is understood, we write $[g_r]$ and  $[g_l]$ respectively.
				\item If $f\in R[x]$ such that $f$ and $g$ induce the same right-function (left-function)  on $R$, i.e., $[g_r] =[f_r] $ ($[g_l] =[f_l] $), then we abbreviate this with $g \quv_r  f$ ($g \quv_l  f$) on $R$.

				\item We define $\PolFun[R]_r=\{[g_r] \mid g\in R[x]\}$, and   
				$$\PrPol[R]_r=\{[f_r] \mid [f_r] \text{ is a permutation of }R \text{ and }   f\in R[x]\}. $$
				Similarly, we define  $\PolFun[R]_l$ and $\PrPol[R]_l$. 	
				\item If $A$ is a subring of $R$, and $f\in A[x]$, then $f$ induces two right-polynomial functions (left-polynomial functions) on $R$ and   $A$ respectively. To distinguish between them we write $[f_r]_R$ and $[f_r]_A$ ($[f_l]_R$ and $[f_l]_A$).
			\end{enumerate}
		\end{definition}

		It is clear that $\PolFun[R]_r$ ($\PolFun[R]_l$) is an additive group with respect to pointwise addition.
		However, unlike    commutative rings, we cannot define pointwise multiplication  on the sets  $\PolFun[R]_r$
		and $\PolFun[R]_l$ since substitution in general is not a homomorphism. Indeed,  one may find $f,g\in R[x]$ and $r_1,r_2\in R$ such that
		$$f_r(r_1)g_r(r_1)\ne h_r(r_1) \text{ and } f_l(r_2)g_l(r_2)\ne h_l(r_2),$$ where
		$h=fg$.  In fact, if we write $f(x)=\sum\limits_j a_jx^j $ and $g(x)=\sum\limits_i x^ib_i $, then we can express	$h $ as $h(x)=\sum\limits_j a_jg(x)x^j$ and  $h(x)=\sum\limits_i  x^if(x) b_i$. Thus, substitutions  in $h$  from right and left  are, respectively, given   by
		\begin{equation}\label{prodsub}
			h_r(c)=\sum\limits_j a_jg_r(c)c^j,
		\end{equation}	
		and
		\begin{equation}
			h_l(c)=\sum\limits_i  c^if_l(c) b_i. 
		\end{equation}

		From now on, we only consider right-polynomial functions. By symmetry, all the properties of right-polynomial functions can be drawn into left-polynomial functions with some suitable changes. We will not use
		the term right and will remove the subscript $r$ indicating it from the notation. In particular, for a polynomial 
	$f=\sum\limits_{i=0}^{n}a_ix^i\in R[x]$, we use $f(b)$ to designate $\sum\limits_{i=0}^{n}a_ib^i $ for every $b\in R$.
		
		We mention here that if   $f,g\in R[x]$, where $g$ is   in the center of $R[x]$, then $(fg)(r)=f(r)g(r)$ for every $r\in R$. Also, if $f$ is in the center of $R[x]$, then $(fg)(r)=g(r)f(r)$ for every $r\in R$. This can be concluded easily by
		following definition in which  we summarize the relation between the multiplication of two polynomials and
		its value on $R$. 
		\begin{definition}\label{noncprod}
			Let $f,g\in R[x]$. Let $f(x)=\sum\limits_{j=0}^{n}a_jx^j$.
			Then 
			\begin{enumerate}
				\item $(fg)(x)=\sum\limits_{j=0}^{n}a_jg(x)x^j$;
				\item  $(fg)(r)=\sum\limits_{j=0}^{n}a_jg(r)r^j$ for every $r\in R$.
			\end{enumerate}
		\end{definition}
		
		\begin{remark}\label{proddefrem}
			The previously defined multiplication of polynomials is the usual multiplication of polynomials which is the same as in the commutative case except that the coefficients do not commute, and therefore
			it is associative. But let us show this by using Definition~\ref{noncprod}. So, if $f,g,h \in R[x]$, then write  $f(x)=\sum\limits_{j=0}^{n}a_jx^j$ and $g(x)=\sum\limits_{i=0}^{m}b_ix^i$.
			Now, $$(fg)(x)=\sum\limits_{j=0}^{n}a_jg(x)x^j=\sum\limits_{j=0}^{n}a_j\sum\limits_{i=0}^{m}b_ix^ix^j=\sum\limits_{j=0}^{n}\sum\limits_{i=0}^{m}a_jb_ix^{i+j}.$$
			Therefore
			$$((fg)h)(x)=\sum\limits_{j=0}^{n}\sum\limits_{i=0}^{m}a_jb_ih(x)x^{i+j}=\sum\limits_{j=0}^{n}a_j(\sum\limits_{i=0}^{m}b_ih(x)x^i)x^j=\sum\limits_{j=0}^{n}a_j(gh)(x)x^j=(f(gh))(x).$$ This also shows that  $((fg)h)(r)=(f(gh))(r)$ for every $r\in R$.
		\end{remark} 
		
		Polynomials that map every element to zero (i.e., induce the zero function) appear widely within investigating polynomial functions, and for
		this we devote the following definition.
		\begin{definition}\label{nulpol}
			\leavevmode
			\begin{enumerate}
				\item A polynomial $f \in R[x]$ is called a null polynomial on $R$ if $f(r)=0$  for every $r\in R$.  In this case, we write $f\quv 0$ on $R$. 	
				\item 	We define  the set $\Null $ as:
				  					  \[\Null=\{ f\in R[x]\mid  f \quv 0 \text{ on } R\}.\]
				 			\end{enumerate}
		\end{definition}
		
		From Definitions~\ref{nulpol} and \ref{noncprod}, one sees without effort the following result.
		\begin{corollary}\label{leftull}
			Let $R$ be a non-commutative ring and let $\Null$ be the set of all null polynomials on $R$. Then
			\begin{enumerate}
				\item $\Null$ is a left ideal of $R[x]$;
				\item $\Null$ is an  ideal of $R[x]$ if and only if $\Null$ is an $R$-right module.
			\end{enumerate}
		\end{corollary}
		\begin{remark}\label{eqrelation}
			For a ring $A$, the left ideal $\Null[A]$ defines an equivalence relation `` $\equiv$ $\mod \Null[A]$'' as the following, for any $f,g\in A[x]$, $f\equiv g \mod \Null[A]$ if and only if  $f-g\in \Null[A]$ which is equivalent to $ f\quv g$ on $A$. Thus, 
			the relation ``$\quv$ on $A$'' is  the same equivalence relation on $A[x]$ defined by $\Null[A]$. Also, there is a one-to-one correspondence between the equivalence classes of $\quv$ and the polynomial
			functions on $A$. Therefore, whenever $A$ is   finite, the following equality holds
			\[|\PolFun[A]|=[A[x]: \Null[A]]\]
			
		\end{remark}
		Throughout whenever $A$ is a ring let $J(A)$ denote its Jacobson radical.
		
		\begin{remark}\label{unitsum}
			Unlike the commutative case, we still do not generally know whether $\Null$ is an ideal (i.e., two-sided ideal).
			However, Werner showed that $\Null$ is an ideal over finite non-commutative rings in which every element can be written
			as a sum of units (for example semisimple rings and local rings)  \cite[Theorem~3.7]{Wernerki}. A result of Stewart~\cite[Lemma~4.5 and Theorem 4.6]{stewarsum} (see also \cite[Proposition~3.6]{Wernerki}) infers that every element  of a finite ring $A$
			is a sum of units if and only if $\mathbb{F}_2\times \mathbb{F}_2$ is not a homomorphic image of the ring $A/J(A) $.  
			
			Furthermore, Frisch showed that when  $R$ is the ring of upper (lower) triangular matrices 
			over a commutative ring $A$,  $\Null $  is an ideal \cite[Theorem 5.2]{frischpolytri}. In particular,
			if $R$ is the non-commutative ring  $UT_n(\mathbb{F}_2)$ of all $n\times n $ upper triangular matrices with entries in $\mathbb{F}_2$ and $n\ge 2$, then $\Null$ is an ideal by Frisch's result. However, not every element of $UT_n(\mathbb{F}_2)$ is a sum of units. For example, the elementary matrix $E_{11}$ (i.e., the matrix with entry $1$ in the first position and zero elsewhere) is not a sum of units, because any sum of units in the ring $UT_n(\mathbb{F}_2)$ gives a matrix whose diagonal elements are all equal to zero or all equal to one depending on the number of units in that sum.  This shows the previously mentioned condition of Werner is only sufficient but not necessary.	
		\end{remark}
		Although we remarked earlier that substitution is not a ring homomorphism from $R[x]$ onto the set of polynomial functions $\PolFun$, as we cannot endow the set $\PolFun$ with pointwise multiplication, we can define multiplication on it utilizing Definition~\ref{noncprod} whenever $\Null$ is an ideal.
		
		\begin{proposition}\label{prowedef}
			Let $R$ be a finite non-commutative ring. Define an operation ``$\cdot$'' on  $\PolFun$ by
			letting $F\cdot F_1=[fg]$, where $f,g\in R[x]$ such that $F=[f]$ and $[g]=F_1$. Then
			$``\cdot"$ is well defined if and only if $\Null$ is a two sided ideal; in this case $\PolFun$ is a ring endowed with multiplication $``\cdot"$ and pointwise addition. 
		\end{proposition}
		
		\begin{proof}
			
			First, we notice by Definition~\ref{noncprod} that, $[fg]\in \PolFun$ for every $f,g\in R[x]$.
			
            ($\Leftarrow$) Assume that $\Null$ is an ideal of $R[x]$. Let $F,F_1\in\PolFun$ and let $f,g\in R[x]$ such that $F=[f]$ and $F_1=[g]$. To show that $F\cdot F_1$ is well defined, we need to show that $[fg]=[f_1g_1]$   for every $f_1,g_1\in R[x]$ such that $F=[f_1 ]$ and $F_1=[g_1]$. Now, if $f_1,g_1\in R[x]$ such that $[f]=[f_1]$ and $[g]=[g_1]$, then $f(r) =f_1(r)$ and $g(r) =g_1(r)$ for every $r\in R$. More precisely, the polynomial $h=f-f_1$ is null on $R$, whence $hg\in \Null$ since $\Null$ is an ideal. Now, writing $f=\sum\limits_{j=0}^{n}a_jx^j$ and $ f_1=\sum\limits_{j=0}^{n}b_jx^j$  yields  in view of Definition~~\ref{noncprod}, for every $r\in R$
			\[(fg)(r)-(f_1g_1)(r)=\sum\limits_{j=0}^{n}a_jg(r)r^j-\sum\limits_{j=0}^{n}b_jg_1(r)r^j=\sum\limits_{j=0}^{n}(a_j-b_j)g(r)r^j=(hg)(r)=0.\] Thus, $(fg)(r)=(f_1g_1)(r)$ for each $r$. But this means $[fg]=[f_1g_1]$. 
			
			($\Rightarrow$) Assume that 	$``\cdot"$ is well defined. In the light of Corollary~\ref{leftull}, we need only show that $\Null$ is an $R$-right module. So, let $f\in \Null$ and let $r\in R$. Then $f\quv 0$  (the zero polynomial) on $R$  or equivalently, $[f]=0$ (the zero function). Thus, if $F$ is the constant polynomial function $r$, then $[fr]=0\cdot F=[0r]=0$ since $``\cdot"$ is well defined. Hence,  $fr\quv 0$ and $fr\in \Null$. 
			For the other part, one can easily see that $\PolFun$ is an additive group with respect to pointwise addition.
			We only check that $``\cdot"$ is associative and leave other properties to the reader.  Let $F_1,F_2,F_3\in \PolFun$, to show that $(F_1 \cdot F_2 )\cdot F_3= F_1 \cdot (F_2 \cdot F_3)$ it is enough to  show that $(fg)h=f(gh)$, where $f,g$ and $h$ any polynomials induce $F_1,F_2$ and $F_3$ respectively. But this is what we have already shown that in Remark~\ref{proddefrem}.
		\end{proof}
		
		\begin{remark}

				Unfortunately,  composition of polynomial functions is not compatible with composition of polynomials on non-commutative rings, that is, one can find two polynomials $f,g\in R[x]$  such that $[f]\circ [g]\ne [f\circ g]$, that is, there exists an $r\in R$ such that $(f\circ g)(r)\ne f(g(r))$. This is because composition depends on multiplication and we already remarked that substitution is not a homomorphism within multiplication of polynomials. For instance,   take $R$ to be the ring of $2\times2$ matrices over the field $\mathbb{F}_2$, and let $a=\begin{bmatrix}
					1 & 0 \\
					0 & 0 
				\end{bmatrix}$ and
				$b=\begin{bmatrix}
					0 & 1 \\
					0 & 1 
				\end{bmatrix}$. 
				For  $f(x)=x^2$ and $g(x)=ax$, we have the polynomial $h(x)=f\circ g(x)=a^2x^2=ax^2$. Then 
				\[h(b)=ab^2=\begin{bmatrix}
					0 & 1 \\
					0 & 0 
				\end{bmatrix}\ne 0= f(g(b)).\] Therefore, $[f]\circ [g]\ne [f\circ g]$.
				\end{remark}
		Because of this incompatibility of the composition of  polynomial functions and the composition of polynomials, we can  neither ensure  that $\PolFun$ is a monoid  nor  that $\PrPol$ is a group. 
		However, we will see that  $\PolFun$ generates a monoid of functions on $R$ and that  $\PrPol$ generates a group of permutations on $R$. 
		For this we need the following definition.
		
		\begin{definition}\label{closdef}
			Let $A$ be a non-empty set with associative operation $*$. For every   non-empty subset $B$ of $A$, we define its closure $\overline{B}$ in $A$ by
			\[\overline{B}=\{a_1*a_2*\cdots *a_n\mid n \in \mathbb{N} \text{ and } a_i\in B \text{ for } i=1,\ldots,n\}.\]  
		\end{definition}
	Given Definition~\ref{closdef}, one easily shows the following useful fact. 
		\begin{fact}\label{closis}
			Let $A$ be a non-empty set with associative operation $*$ and let  $B$ be a non-empty subset  of $A$.
			\begin{enumerate}
				\item If $(A,``*" )$ is a monoid and $1_A\in B$, then $\overline{B}$ is a submonoid of  $A$.
				\item If $(A,``*" )$ is a finite group, then $\overline{B}$ is a subgroup of  $A$.
			\end{enumerate} 
		\end{fact}
		
		Now, considering  $\PolFun$ as a subset of the monoid of all functions on $R$ and $\PrPol$ as a subset of the symmetric group $S_R$ of all permutations on $R$ together with
		Fact~\ref{closis} yields the following.
		\begin{corollary}\label{clospolyf}
			Let $R$ be a finite non-commutative ring  and $``\circ "$  denote the composition of functions.  Then
			\begin{enumerate}
				\item $(\overline{\PolFun},``\circ" )$ is a finite monoid;
				\item $(\overline{\PrPol},``\circ" )$ is a finite group. 
			\end{enumerate}
		\end{corollary} 
		From now on, we call the group (monoid)  $\overline{B}$ the closure group (monoid) of  $B$.
		We conclude this section by indicating that the material in this section is known (see  \cite{frischnon-com,Wernerki}) except Proposition~\ref{prowedef} which we do not know any reference containing it. 
		\section{Dual numbers over finite non-commutative rings}\label{sec3}
		In this section, we expose some elementary properties of the ring of dual numbers in $k$ variables and its polynomial ring. Most of these properties play an essential role in the proofs of the following results in this paper.
		We start by giving a constructive definition of this ring. 
		\begin{definition}\label{CHS001}
			Let $R$ be  a non-commutative ring and let $T$ be the ideal of the polynomial ring $R[x_1,\ldots,x_k]$  generated by the set $\{x_ix_j\mid  i,j\in\{1,\ldots,k\}\}$. We call the quotient   
			ring $R[x_1,\ldots,x_k]/T$ the ring of dual numbers of $k$ variables over $R$.
			We write 
			$\Ralfak$ for  $R[x_1,\ldots,x_k]/T$, where $\beta_i$ denotes $x_i+T$.
		\end{definition}
		\begin{remark}
			Note that every  element of  $R[x_1,\ldots,x_k]/T$ has a unique representation as an $R$-linear combination of $1,\alfa_1,\ldots,\alfa_k$. That is, $\Ralfak$ is a free $R$-algebra with basis $\{1,\alfa_1,\ldots,\alfa_k\}$. We call the coefficient of $1$ the ``constant coefficient''. It follows then that   the polynomial ring ${\Ralfak}[x]$ is a free $R[x]$-algebra with the same basis $\{1,\alfa_1,\ldots,\alfa_k\}$. Also, $R$ is canonically embedded as a subring in 
			$\Ralfak$ by $r\rightarrow r \cdot 1$, and we have
			\[\Ralfak=\{r_0+\sum\limits_{i=1}^{k}r_i\alfa_i\mid  r_0,r_i\in R, \text{ with }\alfa_i \alfa_j=0 \text { for } 1\le i,j\le k \}.\] 
						It follows from this that  every polynomial $f\in {\Ralfak}[x]$ has a unique representation  $f =f_0 +\sum\limits_{i=1}^{k}f_i \alfa_i$, where $f_0,f_1,\ldots,f_k\in R[x]$.
		\end{remark}
		\begin{definition}\label{purpol}
			Let $f\in {\Ralfak}[x]$ and let $f_0,f_1,\ldots,f_k\in R[x]$ be the unique polynomials such that
			$f=f_0+\sum\limits_{i=1}^{k}f_i\alfa_i$. Then we call the polynomial $f_0$ the pure part of $f$.
		\end{definition}
		The following proposition summarizes some properties of $\Ralfak$ whose proof is immediate from Definition~\ref{CHS001}.
		\begin{proposition}\label{dualpro}
			Let $R$ be a non-commutative ring. Then the following statements hold: 
			
			\begin{enumerate}
				\item
				For $a_0,\ldots,a_k,b_0,\ldots,b_k\in R$, we have:
				
				\begin{enumerate}
					\item 
					$(a_0+\sum\limits_{i=1}^{k}a_i\alfa_i)(b_0+\sum\limits_{i=1}^{k}b_i\alfa_i)=a_0b_0+\sum\limits_{i=1}^{k}(a_0b_i+a_ib_0)\alfa_i$;
					\item
					$a_0+\sum\limits_{i=1}^{k}a_i\alfa_i$ is a unit in $\Ralfak$ if and only if 
					$a_0$ is a unit in $R$. In this case,\\ 
					$(a_0+\sum\limits_{i=1}^{k}a_i\alfa_i)^{-1}=a_0^{-1}-\sum\limits_{i=1}^{k}a_0^{-1}a_ia_0^{-1}\alfa_i$.
				\end{enumerate}
				\item
				$\Ralfak$ is a local ring if and only if $R$ is a local ring.
				\item 
				If $R$ is a  ring with  Jacobson radical   $J(R)$ of nilpotency $n$, then  $$J(\Ralfak)=J(R)+\sum\limits_{i=1}^{k}\alfa_i R \text{ of  nilpotency }n+1.$$
				\item If  $R$ is a  ring with center $C(R)$, then
				\begin{enumerate}
					\item	 $C(\Ralfak)=C(R)+\sum\limits_{i=1}^{k}C(R)\alfa_i$;
					\item  $C(R)[x]\subseteq C(\Ralfak)[x]$.
					\end{enumerate}
			\end{enumerate}
		\end{proposition}
		Unlike the commutative case we cannot obtain the same relation as in \cite[Lemma~2.6]{polysev} for the substitution in polynomials. The reason is that the binomial theorem does not hold in non-commutative rings and we need a more general form.  To overwhelm this vital obstacle we first introduce the following definition.
		\begin{definition}\label{assignpol}
			Let $f=\sum\limits_{j=0}^{n}a_jx^j\in R[x]$. Then we assign to $f$ a unique polynomial $\lambda_f(y,z)$ in the non-commuting variables $y$ and $z$ defined by 
			\[\lambda_f(y,z)=\sum\limits_{j=1}^{n}a_jm_j(y,z), \text{ where } m_j(y,z)=\sum\limits_{r=1}^{j}y^{r-1}zy^{j-r}.\]
			We call $\lambda_f$ the assigned polynomial to  $f$ with respect to two non-commuting variables, or just the assigned polynomial. 
		\end{definition}
		\begin{definition}\label{onevarsub}
			Let $f=\sum\limits_{j=0}^{n}a_jx^j\in R[x]$ and   $\lambda_f(y,z)$ be its assigned polynomial in the non-commuting variables $y$ and $z$. For every $a\in R$,  we define the polynomial $\lambda_f(a,z)$ by
			\[\lambda_f(a,z)=\sum\limits_{j=1}^{n}a_jm_j(a,z)=\sum\limits_{j=1}^{n}\sum\limits_{r=1}^{j}a_ja^{r-1}za^{j-r}.\]
		\end{definition}
		\begin{definition}\label{lmadfunction}
			Let $f=\sum\limits_{j=0}^{n}a_jx^j\in R[x]$ and   $\lambda_f(y,z)$ be its assigned polynomial in the non-commuting variables $y$ and $z$. Then
			\begin{enumerate}
				\item the polynomial $\lambda_f$ induces (defines) a function  $F \colon R\times R\longrightarrow R$ by
				substituting the variables 
				\[F(a,b)=\lambda_f(a,b)=\sum\limits_{j=1}^{n}a_jm_j(a,b)= \sum\limits_{j=1}^{n}\sum\limits_{r=1}^{j}a_ja^{r-1}ba^{j-r},\] 
				which we denote by $[\lambda_f(y,z)]$;
				\item for every $a\in R$ the polynomial $\lambda_f(a,z) $ defines a function $F_a \colon   R\longrightarrow R$ by substituting the variable $F_a(b)=\lambda_f(a,b)$, which we denote by
				$[\lambda_f(a,z)]$.
			\end{enumerate}
		\end{definition}
		\begin{remark}\label{notelamda}
			Let $f=\sum\limits_{j=0}^{n}a_jx^j\in R[x]$. Then
			
			\begin{enumerate}
				\item the polynomial $m_j(y,z)=\sum\limits_{r=1}^{j}y^{r-1}zy^{j-r}$ appearing  in Definition~\ref{assignpol} is independent
				of the polynomial $f$.
			
				\item In a similar manner, for every $a\in R$ 
                we can define the polynomial $\lambda_f(y,a) $ by
				\[\lambda_f(y,a)=\sum\limits_{j=1}^{n}a_jm_j(y,a)=\sum\limits_{j=1}^{n}\sum\limits_{r=1}^{j}a_jy^{r-1}ay^{j-r}.\]
				However, in our text, this type of polynomials is less important
				than the polynomial $\lambda_f(a,z) $ of Definition~\ref{onevarsub}.
				\item For every $a,b\in R$ such that $ab=ba$, one easily sees that\label{lamdtoderiv}
				\[\lambda_f(a,b)=\sum\limits_{j=1}^{n}a_jm_j(a,b)= \sum\limits_{j=1}^{n}ja_ja^{j-1}b=f'(a)b.\]  
			\end{enumerate}
		\end{remark}
		Here are straightforward properties of the assigned polynomial which are useful through the coming proofs.
		
		\begin{fact}\label{lamdapro}
			Let $r ,s,w \in R$. Let $f $ and $g \in R[x]$, and let $\lambda_{f}$ and $\lambda_{g}$ be their assigned polynomials respectively. Then
			\begin{enumerate}
				\item $\lambda_{rf+sg}=\lambda_{ rf}+\lambda_{ sg}$; 
				\item $\lambda_{fr+gs}=\lambda_{ fr}+\lambda_{ gs}$;
				\item $\lambda_{ f}(y,z)=0$ if and only if $f$ is constant;
				\item $\lambda_{ f}(y,z)=z$ if and only if $f=x$;
				\item $\lambda_{ f}(0,z)= a_1z$ and $\lambda_{ f}(y,1)=f'(y)$, where $f(x)=\sum\limits_{j=0}^{n}a_jx^j$;
				\item $\lambda_{ f}(y,0)=0$; 
				\item $\lambda_{ f}(r,s\pm w)=\lambda_{ f}(r,s)\pm\lambda_{ f}(r,w)$.
			\end{enumerate}
		\end{fact}
		Now, we are able to state and  prove a general form of \cite[Lemma~2.6]{polysev}, which will play the same role  later in our context.
		
		\begin{lemma}\label{polyeval}
			Let $R$ be a ring and $a,b_1,\ldots,b_k\in R$.
			\begin{enumerate}
				\item
				If $f\in R[x]$ and $\lambda_f $ is its assigned polynomial  
				then
				\[			f(a+\sum\limits_{i=1}^{k}b_i\alfa_i)=f(a)+\sum\limits_{i=1}^{k}\lambda_f(a,b_i)  \alfa_i.
				\]
				\item
				If $f\in {\Ralfak}[x]$ and     $f_0,\ldots, f_k$ are the unique polynomials in $R[x]$
				such that $f =f_0 +\sum\limits_{i=1}^{k}f_i \alfa_i$, then 
				\[			f(a +\sum\limits_{i=1}^{k}b_i\alfa_i)=f_0(a)+ \sum\limits_{i=1}^{k}(\lambda_{f_0}(a,b_i)+f_i(a ))\alfa_i,
				\] where $\lambda_{f_0}$ is the assigned polynomial to $f_0$.\label{2ndcal}
			\end{enumerate}
		\end{lemma}
		\begin{proof}
			(1) By induction and the fact that $\alfa_i\alfa_j=0$ for $1\le i,j\le k$, we have that for every 
			$ a,b_1,\ldots,b_k\in R$, 
			\begin{equation}\label{eqforeval}
				(a+\sum\limits_{i=1}^{k}b_i\beta_i)^j=a^j +\sum\limits_{i=1}^{k}\sum\limits_{r=1}^{j}a^{r-1}b_ia^{j-r}\beta_i=a^j +\sum\limits_{i=1}^{k}m_j(a,b_i)\beta_i,\end{equation} 
			where $m_j(a,b_i)=\sum\limits_{r=1}^{j}a^{r-1}b_ia^{j-r}$.
			
			Now, writing $f=\sum\limits_{j=0}^{n}a_jx^j\in R[x]$ yields, by identity~(\ref{eqforeval}) and Definition~\ref{lmadfunction},
			\[f(a+\sum\limits_{i=1}^{k}b_i\beta_i)=\sum\limits_{j=1}^{n}a_j(a+\sum\limits_{i=1}^{k}b_i\beta_i)^j +a_0=\sum\limits_{j=0}^{n}a_ja^j+\sum\limits_{j=1}^{n}a_j\sum\limits_{i=1}^{k}m_j(a,b_i)\beta_i=f(a)+\sum\limits_{i=1}^{k}\lambda_f(a,b_i)  \alfa_i.\] 
			
			(2) Follows from (1). 
		\end{proof}
		\begin{Notation}
			From now on,  let $R_k$ denote the ring ${\Ralfak}$.
		\end{Notation}
		
		A consequence of Lemma~\ref{polyeval}  is the following result. 
	\begin{corollary}\label{CHSneccondit}
			Let $F\colon R_k\longrightarrow R_k$ be a polynomial function and  
			let $a ,b_{1},\ldots,b_k\in R$.
				Then:
			\begin{enumerate}
				\item The constant coefficient of  $F(a+\sum\limits_{i=1}^{k}b_i \alfa_i)$ depends only on $a$;
				\item The  coefficient of $\alfa_i$ in $F(a+\sum\limits_{i=1}^{k}b_i \alfa_i)$ depends only on $a$ and $b_i$.
			\end{enumerate}
			
		\end{corollary}
	\begin{remark}\label{expresnotef}

	Let $\sigma$ be a permutation of $1,\ldots,k$. Let $a_0,\dots,a_k\in R$ and let $f\in R[x]$ such that 
	$f(a_0+\sum\limits_{i=1}^{k}a_i\alfa_i)=b_0+\sum\limits_{i=1}^{k}b_i\alfa_i$, where $b_0,\dots,b_k\in R$. Then, in view of Lemma~\ref{polyeval} and Corollary~\ref{CHSneccondit}, we easily see that:
	\begin{enumerate}
		\item $f(a_0+\sum\limits_{i=1}^{k}a_i\alfa_{\sigma(i)})=b_0+\sum\limits_{i=1}^{k}b_i\alfa_{\sigma(i)}$;
		\item $f(a_0+\sum\limits_{i=1}^{k}a_{\sigma(i)}\alfa_i)=b_0+\sum\limits_{i=1}^{k}b_{\sigma(i)}\alfa_i$;
		\item $f(a_0+a_i\alfa_j)=b_0+ b_i\alfa_{j}$ for $i=1,\ldots,k$ and $j\ge 1$.
	\end{enumerate} 
 Now, if $f_ 1,\ldots,f_n\in R[x]$ such that 
  
 $[f_n]_{R_k}\circ \cdots\circ [f_1]_{R_k}(a_0+\sum\limits_{i=1}^{k}a_i\alfa_i)=f_n(\cdots(f_1(a_0+\sum\limits_{i=1}^{k}a_i\alfa_i))\cdots)=c_0+\sum\limits_{i=1}^{k}c_i\alfa_i$, where $c_0,\dots,c_k\in R$. Then inductively, we have that:
 \begin{enumerate}[label=(\alph*)]
 	\item $[f_n]_{R_k}\circ \cdots\circ [f_1]_{R_k}(a_0+\sum\limits_{i=1}^{k}a_i\alfa_{\sigma(i)})=c_0+\sum\limits_{i=1}^{k}c_i\alfa_{\sigma(i)}$;
 	\item $[f_n]_{R_k}\circ \cdots\circ [f_1]_{R_k}(a_0+\sum\limits_{i=1}^{k}a_{\sigma(i)}\alfa_i)=c_0+\sum\limits_{i=1}^{k}c_{\sigma(i)}\alfa_i$;
 	\item $[f_n]_{R_m}\circ \cdots\circ [f_1]_{R_m}(a_0+a_i\alfa_j)=c_0+ c_i\alfa_{j}$ ($m=max(i,j)$) for $i=1,\ldots,k$ and $j\ge 1$.
 \end{enumerate}
	\end{remark}
 In the following $F\vert_ A$ stands for the restriction of the function $F$ to $A$.
\begin{corollary}\label{indepfromnumbe}
Let $k>1$ and $F,G\in \overline{\PolFun[R_k]}$. Suppose that $F=[f_n]_{R_k}\circ \cdots\circ [f_1]_{R_k}$ and $G=[g_m]_{R_k}\circ \cdots\circ [g_1]_{R_k}$, where $f_1,\ldots, f_n,g_1,\ldots,g_m\in R[x]$. Then the following statements are equivalent:
\begin{enumerate}
	\item $F=G$;
	\item  $ F\vert_{R_1}=G\vert_{R_1} $; 
	\item $F\vert_{R_j} = G\vert_{R_j} $  for $j=1,\ldots,k-1$.
\end{enumerate}
\begin{proof}
	The implications  (1) $\Rightarrow$ (3) and (3) $\Rightarrow$ (2) are obvious.
	
	(2)$\Rightarrow$ (1):
	assume  that $ F\vert_{R_1}=G\vert_{R_1} $. 
	Let $a_0,\ldots,a_k\in R$ be arbitrary. We want to show that
		\[F(a_0+\sum\limits_{i=1}^{k}a_i\alfa_i)=G(a_0+\sum\limits_{i=1}^{k}a_i\alfa_i).\]
		Suppose that  $F(a_0+\sum\limits_{i=1}^{k}a_i\alfa_i)= b_0+\sum\limits_{i=1}^{k}b_i\alfa_i$ and 
		$G(a_0+\sum\limits_{i=1}^{k}a_i\alfa_i)= c_0+\sum\limits_{i=1}^{k}c_i\alfa_i$, where $b_0,\ldots,b_k,c_0,\ldots,c_k\in R$.
		
		By Remark~\ref{expresnotef},  $F(a_0+a_i\alfa_1)=b_0+b_i\alfa_1$ and  
		$G(a_0+a_i\alfa_1)=c_0+c_i\alfa_1$  for $i=1,\ldots,k$.
But, by assumption, $ F\vert_{R_1}=G\vert_{R_1} $. Therefore,
		 \[b_0+b_i\alfa_1=c_0+c_i\alfa_1 \text{ for }i=1,\ldots,k.\]
		 Thus, $b_i=c_i$ for $i=0,\ldots,k$, whence $F=G$.
\end{proof}
\end{corollary}
		\section{Polynomial functions  on $R_k$}\label{CHSsc3}
		In this section, we characterize null polynomials on the ring $R_k$.  This characterization allows us to determine whether two polynomials are equivalent on $R_k$, that is whether they represent the same polynomial function on $R_k$. Finally, we give a counting formula of the polynomial functions of the ring $R_k$
		by means of the indices of the left ideal  $\Null$ (see Definition~\ref{nulpol}) and  a related  left ideal $A\Null$ defined below.
		
		Next, we define a subset of $\Null$ concerning the assigned polynomials of Definition~\ref{assignpol}.
		\begin{definition} \label{Asnull}
			\leavevmode
			\begin{enumerate}
				\item Let $f \in R[x]$ and let $\lambda_f$ be its assigned polynomial. We call $\lambda_f $  null if $ \lambda_f(a,b)=0$ for every $a,b\in R$. We write  $[\lambda_f(y,z) ]=0$.  	
				\item 	We define $A\Null$ as:
				$A\Null=\{ f\in\Null\mid [\lambda_f(y,z) ]=0\}$.
				\item We define $ \Nulld$ as:
				$\Nulld=\{ f\in\Null\mid f'\in \Null\}$.
				
			\end{enumerate}
		\end{definition}
		\begin{remark}\label{nulremark}
			\leavevmode
			\begin{enumerate}
				\item It is obvious  that $A\Null$ and $\Nulld$  are left ideals of $R[x]$ with $A\Null,\Nulld\subseteq \Null$.
				\item Let $f\in A\Null$. Then the condition $[\lambda_{ f}(y,z)]=0$ implies that
				$\lambda_{ f}(a,1)=f'(a)=0$ for every $a\in R$ (by Fact~\ref{lamdapro}). Hence $f\in\Nulld$.
				Thus, we have the following inclusion
				$$A\Null\subseteq \Nulld\subseteq \Null.$$
				\item When $R$ is commutative the condition on $\lambda_{ f}$ in the definition of $A\Null$ is equivalent to $f'\in \Null$. So, $A\Null =\Nulld$ over a commutative ring $R$. However, when $R$ is a non-commutative ring,  it can  happen that $\Nulld$ contains  $A\Null$ properly (see Example~\ref{nulldnotAnull}).      We will see that the left ideal $A\Null$ plays the same role as $\Nulld$ in the context of commutative rings. More precisely, we will show that most of the results of \cite{Haki,polysev} involving $\Nulld$ still hold here in the context of non-commutative rings by replacing $\Nulld$ by $A\Null$.
				
			\end{enumerate}
			\end{remark}
		\begin{lemma}\label{CHS31} 
			Let $g\in R[x]$. Then:
			\begin{enumerate}
				\item
				$g\in \Null[R_k]$   if and only if 
				$ g\in A\Null $;
				\item
				$g \alfa_i\in \Null[R_k]$   for every $1\le i\le k$ if and only if 
				$g\in\Null$.
			\end{enumerate}
		\end{lemma}
		
		\begin{proof} (1) 
			Let	$a,b_1,\ldots,b_k\in R$. Then by Lemma~\ref{polyeval},
			\[
			g(a+\sum\limits_{i=1}^{k}b_i\alfa_i)=g(a)+\sum\limits_{i=1}^{k}\lambda_g(a,b_i)  \alfa_i.
			\]  So, $g$ is a null polynomial on $R_k$ is equivalent to	\[
			g(a+\sum\limits_{i=1}^{k}b_i\alfa_i)=g(a)+\sum\limits_{i=1}^{k}\lambda_g(a,b_i)  \alfa_i
			=0\] for every $a,b_1,\ldots,b_k\in R$.

			A fortiori, this is  equivalent to $g(a )=0$ and $\lambda_g(a,b_i)= 0$ for all $a ,b_i\in R$ and $i=1,\ldots,k$. But this is equivalent to $g\in \Null$ and  $[\lambda_g(y,z)]= 0$, whence  it is equivalent to $g\in A\Null$.
			
			(2) Straightforward.  
		\end{proof}
		
		\begin{theorem} \label{CHS4} 
			Let $\Null$ and $A\Null$ be as in Definition~\ref{nulpol} and Definition~\ref{Asnull}, respectively. 	
			Then
			\begin{enumerate}
				\item $\Null[R_k]=A\Null+\sum\limits_{i=1}^{k}\Null\alfa_i$;
				\item $\Null[R_k]$ is an ideal of $R_k[x]$ if and only if $A\Null$ and $\Null$ are ideals of $R[x]$.
			\end{enumerate} 
			
		\end{theorem}
		
		\begin{proof}
			(1) By Lemma~\ref{CHS31},  $ A\Null+\sum\limits_{i=1}^{k}\Null\alfa_i\subseteq \Null[R_k]$. For the other inclusion let $f\in \Null[R_k]$. Then $f= f_0 +\sum\limits_{i=1}^{k}f_i \alfa_i$, for some  
			$f_0,\ldots, f_k \in R[x]$.  Let $r\in R$ be arbitrary. Then, by Lemma~\ref{polyeval} and Fact~\ref{lamdapro} since $f\in \Null[R_k]$,
			\[0=f(r)=f_0(r)+ \sum\limits_{i=1}^{k} f_i(r)\alfa_i.\] Thus, since $1,\alfa_1,\ldots,\alfa_k$ is a base of $R_k$ as an $R$-algebra, $ f_i(r)=0$ ($i=0,\ldots,k$) for every $r\in R $, which implies that $f_0,f_1,\ldots,f_k\in \Null$. Now, let $a,b\in R$ be two arbitrary elements. Again by  Lemma~\ref{polyeval} and Fact~\ref{lamdapro}, and what we have  already shown,
			\[0=f(a+b\alfa_1)=\lambda_{ f_0}(a,b).\] 
			Thus, $\lambda_{ f_0}(a,b)=0$ for every $a,b\in R$, whence $[\lambda_{ f_0}(a,b)]=0$. Hence $f_0\in A\Null$. This finishes the proof of the other inclusion. 
			
			(2) 
			Keeping in mind that $A\Null, \Null$  are left ideals of $R[x]$ and $\Null[R_k]$ is a left ideal of $R_k[x]$. ($\Rightarrow$) Assume  that $\Null[R_k]$ is an ideal of $R_k[x]$ and let $f\in A\Null$ and $g\in\Null$.
			Then, by part (1),  the polynomial $h=f+g\alfa_1\in \Null[R_k]$. Hence, $hr= fr+gr\alfa_1 \in \Null[R_k]$ for every polynomial $r\in R_k[x]$ since $\Null[R_k]$ is an ideal, whence in particular, $hr\in \Null[R_k]$ for every $r\in R[x]$. Thus, by the first part  $fr\in A\Null$ and $gr\in \Null$ for every $r\in R[x]$. Therefore, $A\Null$ and $\Null$ are ideals of $R[x]$. The other implication is somewhat similar and left to the reader.
		\end{proof}
		\begin{proposition}\label{sumueqcond}
			Let $R$ be a finite non-commutative ring and $R_k$ be the ring of dual numbers of $k$ variables.
			Then the following statements are equivalent:
			\begin{enumerate}
				\item every element of $R$ is a sum of units;
				\item every element of $R_k$ is a sum of units;
				\item $R /J(R)$ has no factor ring isomorphic to  $\mathbb{F}_2\times \mathbb{F}_2$;
				\item $R_k/J(R_k)$ has no factor ring isomorphic to   $\mathbb{F}_2\times \mathbb{F}_2$.
			\end{enumerate}
		\end{proposition}
		\begin{proof}
			By Remark~\ref{unitsum}, we need to show only (3) $\Leftrightarrow$ (4).
			By Proposition~\ref{dualpro}, $J(R_k)=J(R)+\sum\limits_{i=1}^{k}\alfa_i R$.
			Then one easily sees that
			\[R_k/J(R_k)=(R+\sum\limits_{i=1}^{k}\alfa_i R)/(J(R)+\sum\limits_{i=1}^{k}\alfa_i R) \cong R /J(R). \]	
		\end{proof}
		Combining Remark~\ref{unitsum}, Proposition~\ref{sumueqcond} and Theorem~\ref{CHS4} yields the following
		\begin{corollary}
			Suppose that $R$ (alternatively $R_k$) satisfies the conditions of Proposition~\ref{sumueqcond}. Then
			\begin{enumerate}
				\item $\Null[R_k]$ is an ideal of $R_k[x]$;
				\item $\Null$ and $A\Null$ are ideals of $R[x]$.
			\end{enumerate}
		\end{corollary} 
		From now on we consider a non-commutative ring $R$ in which $\Null$ and $A\Null$ are ideals of $R[x]$ (equivalently $\Null[R_k]$ is an ideal of $R_k[x]$). However, we remark that the results are still true whenever the arguments of the coming proofs involve only the pointwise addition rather than the multiplication 
		``$\cdot$" defined in Proposition~\ref{prowedef}. In such circumstance we deal with $\PolFun,\PolFun[R_k], \Null, A\Null$ and $\Null[R_k]$ as $R$-left modules. 
		
		Because of Lemma~\ref{CHS31} and the definition of the ideal $A\Null$, we rephrase the first part of 	  Theorem~\ref{CHS4}  in the following corollary.
		\begin{corollary}\label{CHSNulleqc}
			Let $f =f_0 +\sum\limits_{i=1}^{k}f_i \alfa_i$, where $f_0,\dots, f_k \in R[x]$. 
			Then the following statements are equivalent:
			\begin{enumerate}
				\item $f\in \Null[R_k]$ (i.e., $f$ is a null polynomial on $R_k$);
				\item $f_0,f_i\alfa_i\in \Null[R_k]$ for $i=1,\ldots,k$;
				\item $[\lambda_{ f_0}(y,z)]=0$ and $f_i\in \Null$ for $i=0,\ldots,k$.
			\end{enumerate}
		\end{corollary}
		Another consequence of Theorem~\ref{CHS4} is the following criterion, which specifies whether two polynomials $f,g\in R_k[x]$ represent the same polynomial function on $R_k$.  
		\begin{corollary}\label{CHS6} \label{CHSGencount}
			
			Let $f =f_0 +\sum\limits_{i=1}^{k}f_i \alfa_i$ and $g=g_0+\sum\limits_{i=1}^{k}g_i \alfa_i $, where 
			$f_0,\ldots, f_k, g_0,\ldots, g_k \in R[x]$.
			
			Then $f \quv g$ on $R_k$ if and only if the following 
			conditions hold:
			\begin{enumerate}
				
				\item
				$[\lambda_{f_0}(y,z)] = [\lambda_{g_0}(y,z)]$;
				\item
				$[f_i]_R= [g_i]_R$ for $i=0,\dots,k$.
			\end{enumerate}
			
			In other words, $f \quv g$ on $ R_k$ if and only if
			the following  congruences hold:
			
			\begin{enumerate}
				
				\item
				$f_0 \equiv g_0  \mod A\Null$;
				\item
				$f_i \equiv g_i  \mod \Null$ for $i=1,\ldots,k$.
			\end{enumerate}
		\end{corollary}
		\begin{proof}
			It is enough to consider the polynomial $h=g-f$ and notice that $g \quv f$ on $R_k$ 
			if and only if $h \in \Null[R_k]$. 
		\end{proof}
		Concerning $A\Null[R]$ and $\Null[R]$ as subgroups of the additive group $R[x]$ gives the following  straightforward result of the Third Isomorphism Theorem of groups. 
		\begin{lemma}\label{CHS2ndisoapli}
			Let $R$ be a finite ring. Then $[R[x]\colon A\Null[R]]=[R[x] \colon \Null[R]][\Null[R] \colon A\Null[R]]$.
		\end{lemma}
		Recall from Definition~\ref{CHSequvfun} that $\PolFun[R_k]$ stands for the set of polynomial functions on $R_k$. In the following proposition, we find a counting formula for $\PolFun[R_k]$ in terms of the indices of the ideals $A\Null, \Null$.  The proof is quite similar to that of \cite[Proposition~3.7]{polysev} with the difference that here we replace  the ideal $\Nulld$ with the ideal $A\Null$.
		\begin{proposition} \label{CHSfirstcountfor}
			The number of polynomial functions on $R_k$ is given by
			\[|\PolFun[R_k]|= \big[R[x]\colon A\Null \big]\big[R[x]\colon \Null\big]^k=[\Null[R] \colon A\Null[R]]|\PolFun|^{k+1}.\]
			
		\end{proposition}
		
		\begin{proof}
			
			Let  $f =f_0 +\sum\limits_{i=1}^{k}f_i \alfa_i$ and $g=g_0+\sum\limits_{i=1}^{k}g_i \alfa_i $
			where $f_0,\ldots ,f_k,g_0,\ldots ,g_k \in R[x]$.  Then by Corollary~\ref{CHS6},
			$[f]_{R_k} = [g]_{R_k}$  if and only if $f_0 \equiv g_0  \mod A\Null$
			and $f_i \equiv g_i  \mod \Null$ for $i=1,\ldots,k$.
			
			Define $\psi\colon \bigoplus\limits_{i=0}^{k} R[x] \longrightarrow \mathcal{F}(\Ralfak)$ 
			by $\psi(f_0,\ldots,f_k)= [f]_{R_k}$, where   $f =f_0 +\sum\limits_{i=1}^{k}f_i \alfa_i$. 
			Then $\psi$ is a group epimorphism of additive groups with  $\ker\psi=A\Null\times\bigoplus\limits_{i=1}^{k}  \Null$ by Theorem~\ref{CHS4}. 
			Therefore, \[|\mathcal{F}({R_k})|=
			[\bigoplus\limits_{i=0}^{k} R[x]\colon A\Null\times \bigoplus\limits_{i=1}^{k}  \Null]=
			[R[x]\colon A\Null][R[x]\colon \Null]^k.\]
			The other equality follows by Lemma~\ref{CHS2ndisoapli} and Remark~\ref{eqrelation}.
		\end{proof}
		\begin{corollary}\label{indecon}
			
			Let $F\in \PolFun$ be fixed. 
			\[[\Null[R] \colon A\Null[R]]=|\{[\lambda_{ f}(y,z)]\mid f\in R[x] \text{ such that } [f]_R=F \}|.\]
		\end{corollary}
		\begin{proof}
			Fix  an arbitrary $F\in \PolFun$. Then
			set $$\mathcal{A}=\{[\lambda_{ f}(y,z)]\mid f\in R[x] \text{ such that } [f]_R=F \},$$ 
			and set $$\mathcal{B}=\{G\in \PolFun[R_k]\mid G=[g_0+\sum\limits_{i=1}^{k}g_i \alfa_i]_{R_k} \text { with }   [g_0]_R=F, \text{ where } g_0,g_1,\ldots,g_k\in R[x]\}.$$
					Then, we claim that $|\mathcal{B}|=|\mathcal{A}|\cdot |\PolFun|^k$. To see this,
			let $G\in \PolFun[R_k]$ be induced by $g=g_0+\sum\limits_{i=1}^{k}g_i \alfa_i$ with $[g]_R=F$, where $g_0,g_1,\ldots,g_k\in R[x]$. Then, by Corollary~\ref{CHSGencount}, the number of such different $G$ corresponds  to the number of different $(k+1)$-tuples $([\lambda_{ g_0}(y,z)],[g_1]_R,\ldots,[g_k]_R)$ with the condition $[g_0]_R=F$. But this condition implies that $[\lambda_{ g_0}(y,z)]$ can be chosen in $|\mathcal{A}|$ ways, whence    the claim follows. 
			
			Now,  we	define a map $\phi\colon \PolFun[R_k] \rightarrow \PolFun$  by letting $\phi (G) =[g_0]_R$. By Corollary~\ref{CHSGencount},
			$\phi$ is well defined. Also,  it is an epimorphism of (additive) groups. Thus, by Proposition~\ref{CHSfirstcountfor},
			\[|\ker \phi| =[\Null[R] \colon A\Null[R]]\cdot |\PolFun|^k. \]
			Thus, $|\phi^{-1} (F)|=[\Null[R] \colon A\Null[R]]\cdot |\PolFun|^k$. Evidently, $\phi^{-1} (F)$ is the set $\mathcal{B}$. Hence, 
			$$[\Null[R] \colon A\Null[R]]\cdot |\PolFun|^k=|\phi^{-1} (F)|=|\mathcal{B}|=|\mathcal{A}|\cdot |\PolFun|^k.$$
			Therefore, $|\mathcal{A}|=[\Null[R] \colon A\Null[R]] $.
		\end{proof}
		\begin{remark}\label{CHSexistmonicn}
			\leavevmode
			\begin{enumerate}
				\item The division algorithm of polynomials over non-commutative rings can be done from two sides, that is, from the left and the right. For, if $f,g\in R[x]$ such that $g$ has a unit leading coefficient, then there exist unique $q,q_1,r,r_1$
				such that $f=q g +r$ with $\deg r< \deg g$ or $r=0$; and $f=gq_1+r_1$   
				with $\deg r_1< \deg g$ or $r_1=0$. We will adopt the first one which is the division from the right.
				\item For any finite non-commutative ring $A$  monic null polynomials not only exist but can be also chosen to be in the center of $A[x]$.   Indeed, if $r\in A$, then there exist two positive integers $n_1(r)>n_2(r)$ such that $r^{n_1(r)}=r^{n_2(r)}$. So, $\prod\limits_{r\in A}(x^{n_1(r)}-x^{n_2(r)})$ is a  monic null polynomial which is in the center of $R[x]$. (see for example~\cite[Proposition~2.2]{Wernerki}). In  particular, if $A$ is our ring $R_k$, then the   monic null polynomial on $R_k$, $\prod\limits_{r\in R_k}(x^{n_1(r)}-x^{n_2(r)})$, is not only in the center of $R_k[x]$ but also admits 
			coefficients in $R$. Furthermore, we can take the least common multiple 
				of the set of polynomials of the form  $(x^{n_1(r)}-x^{n_2(r)})$,  where $r\in R_k$, to get 
				a   monic polynomial of degree $\le \deg \prod\limits_{r\in R_k}(x^{n_1(r)}-x^{n_2(r)})$, which is again in the center.	
			\end{enumerate}
		\end{remark}
		\begin{example}\label{nulldnotAnull}
			Let $R= UT_2(\mathbb{F}_2)$ (the ring of  all $2\times 2$ upper triangular matrices with entries in $\mathbb{F}_2$). Then every $a\in \{ 0, 1,\begin{bmatrix}
				1 & 0 \\
				0 & 0 
			\end{bmatrix}, \begin{bmatrix}
				0 & 0 \\
				0 & 1 
			\end{bmatrix},\begin{bmatrix}
				1 & 1 \\
				0 & 0 
			\end{bmatrix},\begin{bmatrix}
				0 & 1 \\
				0 & 1 
			\end{bmatrix}\}$ is a root of the   polynomial $x^2-x$, and every
			$a\in \{ 
			\begin{bmatrix}
				1 & 1 \\
				0 & 1 
			\end{bmatrix},\begin{bmatrix}
				0 & 1 \\
				0 & 0 
			\end{bmatrix}\}$ is a root of the   polynomial $x^4-x^2$. Thus, $h(x)=lcm(x^2-x, x^4-x^2)=x^4-x^2$ is a   
			null polynomial on $R$, which is in the center of $R[x]$. Furthermore, $h'(x)=4x^3-2x=0\in \Null$. Hence, $h\in \Nulld$. However, by Lemma~\ref{CHS31}, $h\notin A\Null$ since   
			\begin{align*}
				h(\begin{bmatrix}
					1 & 1 \\
					0 & 1 
				\end{bmatrix}+\begin{bmatrix}
					0 & 1 \\
					0 & 1 
				\end{bmatrix}\alfa) &=(\begin{bmatrix}
					1 & 1 \\
					0 & 1 
				\end{bmatrix}+\begin{bmatrix}
					0 & 1 \\
					0 & 1 
				\end{bmatrix}\alfa)^4-(\begin{bmatrix}
					1 & 1 \\
					0 & 1 
				\end{bmatrix}+\begin{bmatrix}
					0 & 1 \\
					0 & 1 
				\end{bmatrix}\alfa)^2\\
				&=\begin{bmatrix}
					1 & 0 \\
					0 & 1 
				\end{bmatrix}-(\begin{bmatrix}
					1 & 0 \\
					0 & 1 
				\end{bmatrix}+\begin{bmatrix}
					0 & 1 \\
					0 & 0 
				\end{bmatrix}\alfa)=\begin{bmatrix}
					0 & 1 \\
					0 & 0 
				\end{bmatrix}\alfa\ne 0.
			\end{align*}
		\end{example}
		The following proposition shows that to  construct a  null polynomial on $R_k$  in the center of $R_k[x]$, it suffices to know a   null polynomial on $R$ in the center of $R[x]$.
		
		\begin{proposition}
			Let $R$ be a finite non-commutative ring and let $f\in R[x]$ be a  null polynomial on $R$   in the center of $R[x]$. Then $f^2$ is a   null polynomial on $R_k$ in the center of $R_k[x]$ for every $k\ge1$.
		\end{proposition}
		\begin{proof}
			Since $f$ is a  null polynomial on $R$ lying in the center of $R[x]$,	we can write $f=\sum\limits_{i=1}^{n}a_ix^i$ for some $a_1,\ldots,a_n\in C(R)$.
			Also, $f$ is  in the center of the polynomial ring $R_k[x]$ by Proposition~\ref{dualpro}. Thus, $f^2$ is in the center of   $R_k[x]$.
			Let $r\in R_k $ and  put $h=f^2$.   Then, by Definition~\ref {noncprod}  since $f$ is in the center,
			\[h(r)=\sum\limits_{i=1}^{n}a_if(r)r^i=f(r)f(r)=(f(r))^2.\]
			Now, write $r=a+\sum\limits_{i=1}^{k}b_i\alfa_i$, where 
			$a,b_1,\ldots,b_k\in R$. Then, by Lemma~\ref{polyeval}  and the fact that $f$ is a null polynomial on $R$,
			\[
			h(r) =	(f(a+\sum\limits_{i=1}^{k}b_i\alfa_i))^2=
			(f(a)+\sum\limits_{i=1}^{k}\lambda_f(a,b_i)  \alfa_i )^2=( \sum\limits_{i=1}^{k}\lambda_f(a,b_i)  \alfa_i)^2=0. 
			\]
			Therefore, $h$ is a null polynomial on $R_k$  for every $k\ge 1$.
		\end{proof}
		In the following proposition, we obtain an upper bound for the minimal degree of a representative of a polynomial function on $R_k$.  
		\begin{proposition}\label{CHSsur}
			Let $R$ be a finite non-commutative ring.	Let $h_1\in {R_k}{[x]}$ and $ h_2 \in R[x]$ be  monic null polynomials on $R_k$ and $R$, respectively, such that $\deg h_1= d_1$ 
			and  $\deg h_2=d_2$.
			
			Then every polynomial function 
			$F\colon R_k\longrightarrow R_k$ 
			is represented by a polynomial  $f =f_0 +\sum\limits_{i=1}^{k}f_i\alfa_i$, where $f_0,\ldots,f_k \in R[x]$ 
			such that $\deg f_0 <d_1$ and  $\deg   f_i < d_2$ for $i=1,\ldots,k$. 
			
			Moreover, if  $F$ is represented by a polynomial $f\in R[x]$ and 
			$h_1\in R[x]$ (rather than in ${R_k}[x]$), then there 
			exists a polynomial $g\in R[x]$ with $\deg g<d_1$, such that 
			$g\quv f$  on $R$ and $[\lambda_{ g}(y,z)]=[\lambda_{ f}(y,z)]$.
		\end{proposition}
		
		\begin{proof} Suppose that $h_1\in {R_k}[x]$ is a monic null polynomial
			on $R_k$ of degree $d_1$. Let $g \in {R_k}[x]$ be a polynomial 
			representing $F$. 
			By the division algorithm,  we have  $g =qh_1+r$ 
			for some $r,q \in R_k[x]$, where $\deg r \le d_1 -1$. Then, since $qh_1\in \Null[R_k]$,
			\begin{equation}\label{subtrighdiv}g(a)=(qh_1)(a)+r(a) =r(a) \text{ for every } a\in R_k.\end{equation}
			Thus, $r(x)$ represents $F$. Then, 
			$r=f_0+\sum\limits_{i=1}^{k}r_i\alfa_i$ 
			for some $f_0,r_1,\ldots,r_k\in R[x]$, and it is obvious that 
			$\deg f_0,\deg r_i \le d_1-1$ for $i=1,\dots,k$. 
			Now let $h_2\in \Null$ be a monic polynomial  of degree $d_2$.  
			Again,  by the division algorithm, we have for $i=1,\ldots,k$, $r_i =q_ih_2+f_i$ 
			for some $f_i ,q_i \in R[x]$, where $\deg f_i \le d_2 -1$.
			Then  by Corollary~\ref{CHS6},  $r_i\alfa_i \quv f_i\alfa_i$ on $R_k$. 
			Thus $f =f_0 +\sum\limits_{i=1}^{k}f_i \alfa_i$ is the desired polynomial. \\
			For the moreover part, the existence of $g\in R[x]$ with $\deg g<d_1$ such that
			$[f]_{R_k} = [g]_{R_k} $   follows by the same argument given in the previous part.
			By Corollary~\ref{CHSGencount},  $[\lambda_{ g}(y,z)]=[\lambda_{ f}(y,z)]$ and $[g]_R=[f]_R$.     
		\end{proof}
		\begin{remark}
			We here mention that	Proposition~\ref{CHSsur} is true in the general case. Indeed, in the proof of Proposition~\ref{CHSsur},  Equation~(\ref{subtrighdiv}) is satisfied as well when $\Null[R_k]$ is just a left ideal.
			Also, if we want to argue concerning the left division argument for the same polynomial $g$ and the monic null polynomial $h_1$, the same argument works provided that $\Null[R_k]$ is an ideal. However, if we suppose that $\Null[R_k]$ is only a left ideal but not a right ideal, then we will face some difficulties. Let us illustrate this and suppose that by the left division algorithm, we have that $g  =h_1 q_1 +r_1 $. So,
			\begin{equation*} g(a)=(h_1q_1)(a)+r_1(a)  \text{ for every } a\in R_k.\end{equation*} Hence, we cannot infer
			that in this case $g(a) =r_1(a)$ for each $a$, unless we have $(h_1q_1)(a)=0$ for each $a\in R_k$. Nevertheless, we can solve this obstacle by requiring that the monic null polynomials in that proposition to be in the center of $R_k[x]$, which they exist by Remark~\ref{CHSexistmonicn}. Note that in the cases previously discussed,  we only consider substituting the variable from the right.
		\end{remark}
		
		\section{ Permutation  polynomials on  $R_k$}\label{Sec04}
		Given a polynomial $f=f_0+\sum\limits_{i=1}^{k}f_i\alfa_i$, we see in this section that determining whether $f$ is a permutation polynomial 
		on $R_k$ depends completely  on the pure part $f_0$ of $f$ and its assigned polynomial $\lambda_{f_0} $.  
	
		\begin{definition}\label{locper}
			We call the function $G\colon R\times R \longrightarrow R$ a local permutation in the second coordinate, if for every $a\in R$ the function $G_a\colon R \longrightarrow R$, $r\rightarrow G(a,r)$, is bijective.
		\end{definition}
		
		When we deal with the function $[\lambda_{ f}(y,z)]$, for some  $f\in R[x]$,  we use the term   `` variable  $z$ '' instead of ``  second coordinate''.
		\begin{lemma}\label{perpurcof}
			Let  $R$ be a finite ring and   
			   let $f_0,\ldots,f_k \in R[x]$ be arbitrary. Suppose that $f_0+\sum\limits_{i=1}^{k}f_i\alfa_i$ is a permutation polynomial on $R_k$. Then  $f_0$ is a permutation polynomial  on $R$  and  $[\lambda_{f_0}(y,z)] $ is a local permutation in the variable $z$.    
		\end{lemma}
		\begin{proof}	
			Let $f_0,\ldots,f_k \in R[x]$ and put $f=f_0+\sum\limits_{i=1}^{k}f_i\alfa_i$.
			Suppose that $f$ is a permutation polynomial on $R_k$ and let $a,b_1,\ldots,b_k\in R$.
			Then, by Lemma~\ref{polyeval}, 	\[			f(a+\sum\limits_{i=1}^{k}b_i\alfa_i)=f_0(a )+ \sum\limits_{i=1}^{k}(\lambda_{f_0}(a,b_i)+f_i(a))\alfa_i.
			\] So, the constant coefficient   of $f(a+\sum\limits_{i=1}^{k}b_i\alfa_i)$ is completely determined by  the value of $f_0$ at $a$. Hence, $f_0$ is surjective on $R$ since otherwise $f$ cannot be surjective on $R_k$.
			Thus $f_0$ is a permutation polynomial on $R$ since $R$ is finite.
					Now, let $a\in R$ be arbitrary. To show that $[\lambda_{f_0}(a,z)] $ is surjective, let $c\in R$. Then, since $f$ is a permutation polynomial on $R_k$, there exist $d,b_1,\ldots,b_k\in R $ (i.e., $d+\sum\limits_{i=1}^{k}b_i\alfa_i\in R_k$) such that  $$f(d+\sum\limits_{i=1}^{k}b_i\alfa_i)=f_0(a)+(f_1(a)+c)\beta_1.$$ But then by Lemma~\ref{polyeval},
			\[
			f_0(d)+\sum\limits_{i=1}^{k}(\lambda_{f_0}(d,b_i)+f_i(d)) \alfa_i=f_0(a)+(f_1(a)+c)\beta_1.
			\]
			Hence, $f_0(d)=f_0(a)$ and  $\lambda_{f_0}(d,b_1)+f_1(d)= f_1(a)+c$ since $R_k$ is an $R$-algebra with base $1,\beta_1,\ldots,\beta_k$. So, $a=d$ since $f_0$ is a permutation polynomial on $R$, 
			whence $\lambda_{f_0}(a,b_1)=c$. Thus,
			$[\lambda_{f_0}(a,z)] $ is surjective, so it is bijective since $R$ is finite. Thus, $[\lambda_{f_0}(y,z)] $ is a local permutation in $z$ by Definition~\ref{locper}.
		\end{proof}
		
		\begin{lemma}\label{equvper}
        Let  $R$ be a finite ring and let $f_0\in R[x]$. Suppose that $f_0$ is a permutation polynomial  on $R$  and  $[\lambda_{f_0}(y,z)] $ is a local permutation in the variable $z$.  Then $  f_0+ \sum\limits_{i=1}^{k}f_i \alfa_i$ is a permutation polynomial on $R_k$ for every $f_1,\ldots,f_k \in R[x]$. In particular, $f_0$  is a permutation polynomial on $R_k$.
        
        
		\end{lemma}
		\begin{proof}
        Let   $f_0\in R[x]$ be   a permutation polynomial  on $R$  and assume that $[\lambda_{f_0}(y,z)] $ is a local permutation in  $z$.
       Then, for arbitrarily   $f_1,\ldots,f_k\in R[x]$, put $f=f_0+\sum\limits_{i=1}^{k} f_i\alfa_i.$ We show that $f$ is a permutation polynomial on $R_k$.
			Since $R_k$ is finite, it is enough to show that $f$ is injective on $R_k$, that is $f$ induces an injective function on $R_k$. 
			Let $a,b_1,\ldots,b_k,c,d_1,\ldots,c_k\in R$ such that $f(a+\sum\limits_{i=1}^{k}b_i\alfa_i)=f(c+\sum\limits_{i=1}^{k}d_i\alfa_i)$.
			Then by    Lemma~\ref{polyeval},
			\[
			f_0(a)+\sum\limits_{i=1}^{k}(\lambda_{f_0}(a,b_i)+f_i(a)) \alfa_i= 
			f_0(c)+\sum\limits_{i=1}^{k}(\lambda_{f_0}(c,d_i)+f_i(c)) \alfa_i.  \]
			Thus, we have that $f_0(a)=f_0(c)$ and $\lambda_{f_0}(a,b_i)+f_i(a)=\lambda_{f_0}(c,d_i)+f_i(c)$ for $i=1,\ldots,k$. Hence,
			since $f_0$ is a permutation polynomial on $R$, $a=c$, whence  $\lambda_{f_0}(a,b_i)=\lambda_{f_0}(a,d_i)$ for $i=1,\ldots,k$. Thus, $b_i=d_i$ for $i=1,\ldots,k$ since $[\lambda_{f_0}(a,z)]$  is injective.   Therefore, $f$ is a permutation polynomial on $R_k$. The last assertion follows by choosing $f_1=\cdots=f_k=0$.
		\end{proof} 
		Now we are in a position to give a characterization of permutation polynomials on  the ring $R_k$.
		\begin{theorem}\label{Cherper} 
			Let  $R$ be a finite non-commutative ring. Let $f =f_0+ \sum\limits_{i=1}^{k}f_i \alfa_i$, 
			where $f_0,\ldots,f_k \in R[x]$, and let $\lambda_{f_0} $ be the assigned polynomial to $f_0$ in the non-commuting variables $y$ and $z$. Then 
			the following statements are equivalent:
			\begin{enumerate}
				\item 	$f$ is a permutation polynomial
				on $R_k$;
				\item 	$f_0$ is a permutation polynomial
				on $R_k$;
				
				\item $f_0$ is a permutation polynomial on $R$ and $[\lambda_{f_0}(a,z)] $ is  surjective  for every $a\in R$;
				\item $f_0$ is a permutation polynomial on $R$ and $[\lambda_{f_0}(a,z)] $ is injective for every $a\in R$;
				\item $f_0$ is a permutation polynomial on $R$ and $[\lambda_{f_0}(y,z)] $ is a local  permutation in    the variable $z$.	\end{enumerate}
			
		\end{theorem}
		\begin{proof}
			By Lemma~\ref{perpurcof} and Lemma~\ref{equvper}, the statements (1), (2), and (5) are equivalent.
			Since $R$ is finite, the statements (3), (4), and (5) are equivalent.
		\end{proof}
		
		\begin{remark}
			It is worth here to mention that	for the ring of $n\times n$  matrices  over a finite local ring $R$, $M_n(R)$, Brawley~\cite{matringpolybr} characterized scalar permutation polynomials of $M_n(R)$  by putting conditions not only on these polynomials but also on their assigned polynomials. More explicitly 
			he proved the following criterion  ~\cite[Theorem~2]{matringpolybr}:
			
			Let $R$ be a finite commutative local ring with Maximal ideal $M\ne \{0\}$. Let $f\in R[x]$ and let $\bar{f}\in \mathbb{F}_q[x]$ be the image of $f$ in $\mathbb{F}_q[x]$, where $\mathbb{F}_q=R/M$. Then
			$f$ is a permutation polynomial on $M_n(R)$ if and only if 
			\begin{enumerate}
				\item $\bar{f}$ is a permutation polynomial on $M_n(\mathbb{F}_q)$, and
				\item for every matrix $A\in M_n(\mathbb{F}_q)$, the function $[\lambda_{\bar{f}}(A,z)]$ is a permutation of
				$M_n(\mathbb{F}_q)$. 
			\end{enumerate}
		\end{remark}
		Theorem~\ref{Cherper} shows that the criterion 
		to be a permutation polynomial on $R_k$ depends only
		on $f_0$. As a consequence, we have the 
		following corollary.
		
		\begin{corollary}\label{CHSPPfirstcoordinate} Let $R$ be a finite non-commutative ring.
			Let $f =f_0 +\sum\limits_{i=1}^{k}f_i \alfa_i$,
			where $f_0,\dots,f_k \in R[x]$. 
			Then the following statements are equivalent:
			\begin{enumerate}
				\item $f$ is a permutation polynomial
				on $R_k$; 
				\item $f_0+f_i\alfa_i$ is a permutation polynomial on $R[\alfa_i]$ for every $i\in \{1,\ldots,k\}$;
			
				\item $f_0$ is a permutation polynomial on $R[\alfa_i]$ for every $i\in \{1,\ldots,k\}$;
				\item $f_0+\sum\limits_{i=1}^{j}f_i\alfa_i$ is a permutation polynomial on $R_l$ for every $1\le j\le k$ and 
				$l\ge j$;
				
				\item $f_0$ is a permutation polynomial on $R_j$ for every $j\ge 1$.
			\end{enumerate}
		\end{corollary}
		Another consequence of Theorem~\ref{Cherper} is the following.
		\begin{corollary}\label{lamdasur}
			Let $f =f_0 +\sum\limits_{i=1}^{k}f_i \alfa_i$ be a permutation polynomial on $R_k$,
			where $f_0,\dots,f_k \in R[x]$. Then the function $[\lambda_{f_0}(y,z)]\colon R\times R \longrightarrow R $ is surjective.  
		\end{corollary}
		\begin{remark and question}\label{isitredundant}
			\leavevmode
			\begin{enumerate}
				\item If $R$ is a commutative ring, then it will not be hard to see that the condition on $\lambda_{ f_0}(y,z)$ ($f_0\in R[x]$) in Theorem~\ref{Cherper}  is equivalent to $f_0'$ maps $R$  to its group of units (see also Remark~\ref{notelamda}-(\ref{lamdtoderiv})). But then ~\cite[Theorem 4.1]{polysev} becomes a special case of Theorem~\ref{Cherper}. 
				\item In  the special case $R$ is a local commutative ring that is not a field (hence in the case  $R$ is a direct sum of local rings that are not fields), the condition on $f_0'$ is redundant, that is 
				$f_0$ is a permutation polynomial on $R_k$ if and only if $f_0$ is a permutation polynomial on $R$ (\cite[Proposition~4.7]{polysev}). We notice here that   over the finite field $\mathbb{F}_q$ there exist  $f\in \mathbb{F}_q$ a permutation  polynomial  on  $\mathbb{F}_q$ that is not a permutation polynomial on  ${(\mathbb{F}_q)}_{\tiny{k}}$ (for example $x^q$). However, we can always find a polynomial 
				$g\in {\mathbb{F}_q}$ such that $[g]_{\mathbb{F}_q}=[f]_{\mathbb{F}_q}$ and $g$ is a permutation polynomial on  ${(\mathbb{F}_q)}_{\tiny{k}}$ (see for instance \cite[Lemma~4.9]{Haki} or \cite[Lemma~4.10]{polysev}).
				\item The previous point motivates us to ask  the following  question in  the non-commutative case:
				
				Is the condition  on $\lambda_{ f_0}(y,z)$  in Theorem~\ref{Cherper}   redundant? 
				
				Or equivalently:
				
				Given $f_ 0\in R[x]$. Is  $f_0$    a permutation polynomial on $R$  if and only if  $f_0$ is a permutation polynomial on $R_k$?\label{qos}
				
			\end{enumerate}
		\end{remark and question}
	In the following, we show that the set of polynomials of the form 
	$x+\sum\limits_{i=1}^{k}f_i\alfa_i$, where $f_i\in R[x]$ for $i=1,\ldots,k$, is an abelian group with respect to composition of polynomials for any ring (not necessarily commutative) with unity $1\ne 0$. Furthermore, the set of functions  induced by this group of polynomials is an abelian group of permutations on $R_j$ for every $j\ge k$. Moreover, the composition of permutations in this induced group is compatible  with the composition of their defining polynomials.
	\begin{proposition}\label{nicegroups}
	Let $R$ be a ring with $1\ne 0$ and $k\ge 1$. Set
	\[P_{x\,,k}=\{x+\sum\limits_{i=1}^{k}f_i\alfa_i\mid f_i\in R[x] \text{ for } i=1,\dots,k\}, \text{ and}\]
		\[\mathcal{P}_{x\,,k}=\{[f]_{R_k}\mid f \in P_{x\,,k}\}.\]
		Then 
		\begin{enumerate}
			\item $P_{x\,,k}$ is an abelian group with respect to composition of polynomials; 
			\item $P_{x\,,k}\lhd P_{x\,,j}$ for every $j>k$;
			\item $\mathcal{P}_{x\,,k}$ is an abelian group with respect to composition of functions. Further, for every  $F_1,F_2\in \mathcal{P}_{x\,,k}$ and $f,g\in P_{x\,,k}$
			such that $F_1=[f]_{R_k} $ and $F_2=[g]_{R_k} $, we have that
			\[F_1\circ F_2=[f]_{R_k}\circ [g]_{R_k}=[f \circ g]_{R_k};\]	
			\item $\mathcal{P}_{x\,,k}$ is embedded normally in $\mathcal{P}_{x\,,j}$ for every $j>k$;
		 \item    $\mathcal{P}_{x\,,k}$ is embedded in  $\overline{\PrPol[R_j]} $ for every $j\ge k$, whenever $R$ is finite. In this case, $|\mathcal{P}_{x\,,k}|=|\PolFun|^k$.
		 		\end{enumerate}
	\end{proposition}
\begin{proof}
(1) First, we show  that $P_{x\,,k}$ is closed with respect to composition. Let $f,g\in  P_{x\,,k}$. Then $f=x+\sum\limits_{i=1}^{k}f_i\alfa_i$ and $g=x+\sum\limits_{i=1}^{k}g_i\alfa_i$, where $f_1,\ldots,f_k,g_1,\ldots,g_k\in R[x]$. So, we can write 
$f_i=\sum\limits_{l=0}^{n_i}a_{l\.i}x^l$, where $a_{l\.i}\in R$ for $l=0,\ldots,n_i$ and $i=1,\ldots,k$.
Consider, 	\begin{align*}
	f\circ g &=g +\sum\limits_{r=1}^{k}f_r(g)=g+\sum\limits_{r=1}^{k}f_r(x+\sum\limits_{i=1}^{k}g_i\alfa_i)\alfa_r\\ 
	& =g+\sum\limits_{r=1}^{k}\sum\limits_{l=0}^{n_r}a_{l\,r}(x+\sum\limits_{i=1}^{k}g_i\alfa_i)^l\alfa_r=g+\sum\limits_{r=1}^{k}\sum\limits_{l=0}^{n_r}a_{l\,r}x^l \alfa_r \quad (\text{ since  }\alfa_i\alfa_r=0)\\
	&=x+\sum\limits_{i=1}^{k}g_i\alfa_i +\sum\limits_{r=1}^{k}f_r\alfa_r=x+\sum\limits_{i=1}^{k}(g_i+f_i)\alfa_i\in P_{x\,,k}.  
 	\end{align*}
 Similarly, we have that \begin{equation}\label{forcompat} g\circ f=x+\sum\limits_{i=1}^{k}(f_i+g_i)\alfa_i= x+\sum\limits_{i=1}^{k}(g_i+f_i)\alfa_i =f\circ g.
 \end{equation} Evidently, $x$ is the identity of $ P_{x\,,k}$ and composition of polynomial is associative.
 Also, if $f=x+\sum\limits_{i=1}^{k}f_i\alfa_i\in P_{x\,,k}$, then it will not be hard  to see that $h=x-\sum\limits_{i=1}^{k}f_i\alfa_i\in P_{x\,,k}$ is the inverse of $f$. Therefore, $ P_{x\,,k}$ is an abelian group.
 
 (2) Follows from (1) since $ P_{x\,,k}$ is contained in $ P_{x\,,j}$  for every $j>k$.
 
 (3) First, we show the operation  is closed on  $ \mathcal{ P}_{x\,,k}$. So, let $F_1,F_2\in \mathcal{ P}_{x\,,k}$. Then $F_1=[f]_{R_k}$ and $F_2= [ g]_{R_k}$ for some $f,g\in P_{x\,,j}$ such that $f= x+\sum\limits_{i=1}^{k}f_i\alfa_i$ and $g= x+\sum\limits_{i=1}^{k}g_i\alfa_i$ where
 $f_i,g_i\in R[x]$ for $i=1,\ldots,k$. Now, let $a_0,\ldots a_k \in R$ and consider
 \begin{align*}
 F_1\circ F_2 &(a_0+\sum\limits_{i=1}^{k}a_i\alfa_i)=F_1(F_2(a_0+\sum\limits_{i=1}^{k}a_i\alfa_i))=F_1(g(a_0+\sum\limits_{i=1}^{k}a_i\alfa_i))\\
 & =F_1(a_0+\sum\limits_{i=1}^{k}(a_i+g_i(a_0))\alfa_i) \text{ (by Lemma~\ref{polyeval})}\\
 & =f(a_0+\sum\limits_{i=1}^{k}(a_i+g_i(a_0))\alfa_i) =	a_0+\sum\limits_{i=1}^{k}(a_i+g_i(a_0))\alfa_i + \sum\limits_{i=1}^{k} f_i(a_0)\alfa_i\\
 & =a_0+\sum\limits_{i=1}^{k} a_i\alfa_i + \sum\limits_{i=1}^{k} (f_i(a_0)+g_i(a_0))\alfa_i\\
 &=
 (f\circ g)(a_0+\sum\limits_{i=1}^{k} a_i\alfa_i) \text{ (by Equation ~(\ref{forcompat}))}.	\end{align*}
 Thus, $F_1\circ F_2=[f\circ g]_{R_k} \in \mathcal{ P}_{x\,,k}$ since $f\circ g=x+\sum\limits_{i=1}^{k} (f_i +g_i)\alfa_i\in P_{x\,,k}$.
 This also, shows that $[f]_{R_k}\circ [g]_{R_k}=[f\circ g]_{R_k}$, i.e., composition of functions is compatible with composition of their defined polynomials. Further, by Equation~(\ref{forcompat}), we see easily that $F_1\circ F_2=F_2\circ F_1$. Evidently, composition of functions is associative.  Further, it is clear that $[x]_{R_k}$ is the identity element of $\mathcal{ P}_{x\,,k}$.  Finally, it is not difficult to see that if $F=[x+\sum\limits_{i=1}^{k}f_i\alfa_i]_{R_k}$, then $F$ is invertible and its inverse is the function
 $[x-\sum\limits_{i=1}^{k}f_i\alfa_i]_{R_k}$.
 
 (4) Let $j>k$ and  define a map \[\alpha:
 \mathcal{ P}_{x\,,k} \longrightarrow \mathcal{ P}_{x\,,j},\quad
 F\mapsto [x+\sum\limits_{i=1}^{k}f_i\alfa_i]_{R_j}, \text{ where } F=[x+\sum\limits_{i=1}^{k}f_i\alfa_i]_{R_k}.
 \] By Corollary~\ref{CHSGencount}, $\alpha$ is well defined. Further, by the compatibility property, we have that  $\alpha$ is a homomorphism. Now, if $\alpha (F) =[x+\sum\limits_{i=1}^{k}f_i\alfa_i]_{R_j}$ is the identity function on $R_{j}$, then $[x+\sum\limits_{i=1}^{k}f_i\alfa_i]_{R_k}=F$ is the identity on $R_k$
 since $[x+\sum\limits_{i=1}^{k}f_i\alfa_i]_{R_k}$ is the restriction of $[x+\sum\limits_{i=1}^{k}f_i\alfa_i]_{R_j}$ to $R_k$. Thus, $\ker \alpha$ contains only the identity.
 Therefore, $\alpha$ is a monomorphism, and $\mathcal{P}_{x\,,k}$ is embedded   in $\mathcal{P}_{x\,,j}$ for every $j>k$. Evidently, $\alpha (\mathcal{P}_{x\,,k})$ is a normal subgroup of $\mathcal{P}_{x\,,j}$ (being a subgroup of an abelian group).

 (5) Follows from (4) and Corollary~\ref{CHSGencount}.
 \end{proof} 
\begin{remark}\label{hascomp}
	We have already seen in the previous proposition that the set $\mathcal{P}_{x\,,k}$ is a group of permutations on $R_k$. So, evidently  $\mathcal{P}_{x\,,k}$ is contained in $ \PrPol[R_k]$. Thus, $\mathcal{P}_{x\,,k}$ is a subgroup of the closure group $ \overline{\PrPol[R_k]}$. We  will see later in Section~\ref{6}   
	in the finite commutative case that $\mathcal{P}_{x\,,k}$ has a complement in the group $  \PrPol[R_k]$ (since in this case $ \overline{\PrPol[R_k]}=\PrPol[R_k]$)
\end{remark}
	
		In the following, we find the cardinality of the set $ \PrPol[R_k]$ in terms of the number of polynomial functions
		on the ring $R$ and the number of pairs $([g]_R,[\lambda_{ g}(y,z)])$ such that $g\in R[x]$ is a permutation polynomial on $R$ and $\lambda_{ g}(y,z)$ maps $R\times R$ onto $R$.
		\begin{proposition} \label{CHSGeperncount}
			Let $R$ be a finite non-commutative ring. 
			Let $L$ designate   the number of pairs of functions $(F,H)$ such that  
			\begin{enumerate}
				\item $F\colon R\longrightarrow R $ is bijective;
				\item $H\colon R\times R\longrightarrow R$ is a local permutation in the second coordinate;
			\end{enumerate}
			occurring as $([f]_R,[\lambda_{ f}(y,z)])$ for some $f\in R[x]$.
		
				Then  the number  of polynomial permutations 
			on $R_k$  is given by \[|\PrPol[R_k]|=L\cdot |\PolFun[R]|^k.\] 
		\end{proposition}
		\begin{proof}
			Given $G\in \PrPol[R_k]$. Then by definition  there exist $g_0,\ldots,g_k\in R[x]$ such that $G $ is induced by the polynomial $g =g_0 +\sum\limits_{i=1}^{k}g_i \alfa_i$. By Theorem~\ref{Cherper}, $[g_0]_R\colon R\longrightarrow R$ is bijective,   $[\lambda_{ g_0}(y,z)]$ is a local permutation in $z$,  and $ [g_i]_R\text{ is arbitrary in }$ $\PolFun \text{ for }i=1,\ldots,k$.
			Thus the result  
			follows by Corollary~\ref{CHSGencount}. 
		\end{proof}
		In the commutative case, the number $L$ is shown to be the numbers of pairs $([f]_R,[f']_R)$ for some permutation $f\in R[x]$ such that $[f']_R$ is a unit-valued polynomial function (i.e., maps   $R$ into its units). In particular, for  the case of the finite field $\mathbb{F}_q$ of  $q$ elements, $L=q!(q-1)^q$ (see~\cite[Proposition~4.11]{polysev}).
		In the next section,  we will obtain different combinatorial descriptions of the number  $L$ of Proposition~\ref{CHSGeperncount}.
		\subsection{Permutation polynomials over finite non-commutative  local chain rings}
		\leavevmode
		The purpose of this subsection is to answer the question of  Remark and Question~\ref{isitredundant} affirmatively for a wide class of finite non-commutative chain rings. For this aim, we recall some facts
		about finite local rings in general and in particular about finite chain rings,  the case of interest.
		
		A finite ring $R$ is called a local ring if the set $M$ of all zero-divisors of $R$ is an ideal (two-sided ideal) of $R$.  In this case, $M$ is the unique maximal ideal of $R$; and there exists a minimal positive integer $N$ such that  $M^N=\{ 0\}$  called the nilpotency index of $M$. Also, the characteristic of the ring  $Char (R)$ is a power of some  prime $p$, that is $Char (R)=p^c$ ($1\le c\le N$); and $R/M =\mathbb{F}_q$ where $q=p^w$ ($w\ge 1$). We notice here that if $c=N$, then $R$ is a commutative ring   (see for example \cite{Raghfinite}).
		Furthermore,  if the lattice of left ideals (equivalently of right ideals) is a chain, $R$ is called a chain 
		ring. It follows then that $M^i=t^iR=Rt^i$  for some element $t\in M\setminus M^2$ ($i=0,1,\ldots,N)$. In particular, fixing an element $a\in R$ and ($1\le i<N$), we have
		\[at^i=t^ia_1 \text{ for some } a_1\in R. \]
		Throughout, we use  $p$  also to designate the element  $\underbrace{1_R+\cdots+1_R}_{ p \text{ terms}}$, and by the ramification index of the finite local chain ring we mean the smallest positive integer $e$ such that $p\in M^e\setminus M^{e+1}$. 
		Also, we mention here that for a finite ring $R$, being a chain ring is equivalent to being a local principal ideal ring. The above-mentioned properties of finite chain rings can be found in \cite{NechaSurvey} and the references therein, and for more recent results we refer the reader to \cite{sami}.  
		
		As we deal with non-commutative rings, we consider implicitly the case $ c<N$, where $N$ is the nilpotency index of the maximal ideal $M$  and    $Char (R)=p^c$.
		
		Recall that a ring $A$ is a semi-commutative ring whenever $a,b\in A$ with $ab=0$ implies that $aAb=0$.
		Then, we have the following lemma which is essential in the forthcoming proofs.
		\begin{lemma}\label{togennecha}
			Let $R$ be a finite chain ring and let $f\in R[x]$. The following statements hold: 
			\begin{enumerate}
				\item $R$  is semi-commutative;
				\item $f(a+m)=f(a) +\lambda_{f}(a,m)$ for every $a,m\in R$  with $m^2=0$. 
			\end{enumerate}
		\end{lemma}
		\begin{proof}
			Let $R$ be a finite chain ring with maximal ideal $M=tR=Rt$.
			
			(1)   We want to show that for every $a,b\in R$ such that $ab=0$, it follows then $adb=0$ for every $d\in R$. Without loss of generality, we may assume that $a,b\ne 0$. So, if $ab=0$, then $a,b\in M$ since $R$ is local.  By the properties of the ideal $M$, there exist $k_1,k_2\ge 1$ such that $a\in  M^{k_1}\setminus M^{k_1+1}$ and $b\in  M^{k_2}\setminus M^{k_2+1}$. So, $a=a_1t^{k_1}$ and $b=t^{k_2}b_1$ for some $a_1,b_1\in R$. Hence, $a_1,b_1$ are units in $R$ since otherwise we will have $a\in M^{k_1+1}$ and $b\in   M^{k_2+1}$ which is not possible by our choice of $k_1$ and $k_2$.
			Now, $$0=ab=a_1t^{k_1}t^{k_2}b_1=a_1t^{k_1+k_2}b_1,$$
			whence $$t^{k_1+k_2}=0.$$
			Then, consider an arbitrary element $d\in R$. Since $M^{k_1}=Rt^{k_1}=t^{k_1}R$, there exists $d_1\in R$
			such that $dt^{k_2}=t^{k_2}d_1$. Therefore,
			\[a d b=a_1t^{k_1}dt^{k_2}b_1=a_1t^{k_1}t^{k_2}d_1b_1=a_1t^{k_1+k_2}d_1b_1=0. \]
			(2) Let $m\in R$ with $m^2=0$, then by the first assertion, $m a^j m =0$  for every  $a\in R$ and every $j\ge 1$. 
			Then it will not be hard to see that  $(a+m)^j=a^j + \sum\limits_{r=1}^{j}a^{r-1}ma^{j-r} =a^j + m_j(a,m)$, where $m_j(a,m)=\sum\limits_{r=1}^{j}a^{r-1}ma^{j-r}$. 
			Now, the rest of the proof is similar to the argument given in the proof of Lemma~ \ref{polyeval}, and we leave the details to the reader.
		\end{proof}
		We can replace  the polynomial $\lambda_{ f}$ by $f'$ in the second assertion of the previous lemma by 
		requiring that  $a$ and  $m$ are commutable.  The following lemma shows this in general.
		\begin{lemma}\label{calfor2z}
			Let $R$ be a finite non-commutative local ring and let $f\in R[x]$. Then
			$f(a+m)=f(a) +f'(a)m$ for every $a,m\in R$  such that $am=ma$ and $m^2=0$.
		\end{lemma}
		\begin{proof}
			Let $a,m\in R$ such that $am=ma$ and $m^2=0$. Then, it follows that for $j\ge 1$,
			$(a+m)^j=a^j +ja^{j-1}m$. So, if $f=\sum\limits_{j=0}^{n}a_jx^j$, then
			\[f(a+m)=\sum_{j=0}^{n}a_j(a+m)^j=a_0 +\sum_{j=1}^{n}a_ja^j  +\sum_{j=1}^{n}ja^{j-1}m=f(a)+f'(a)m.\]
		\end{proof}
		\begin{lemma}\label{casec>2}
			Let $R$ be a finite chain ring of characteristic $p^c$ with $c>2$ and maximal ideal $M$ of  nilpotency index $N$, and let $a\in R$. If $pa=0$, then $a^2=0$. 
		\end{lemma}
		\begin{proof}
			Let $M=tR$ and let $e$ be the ramification index of $R$.
			Since $pa=0$, $Char(R)=p^c$ and $c>2$, we have that $a\in M$.
			If $e$ is the ramification index of $R$, then $ p=u_1t^e$ for some unit $u_1$. Further, we can argue as in the proof of Lemma~\ref{togennecha}  to find  another  three units  $u_2, u_3, u_4\in R$ and a positive integer $l$ such that, $a=t^lu_2$, $a^2= t^{2l}u_3$ and $p^{c-1}= u_4t^{ec-e}$. The proof will be finished by showing that $2l\ge N$.
			
			Then $0=pa=u_1t^{e+l}u_2 $ implies that $t^{e+l}=0$ and hence \begin{equation}\label{steq}
				l+e\ge N 
			\end{equation} 
			Since $Char(R)=p^c$,  we have $0\ne p^{c-1}=  u_4t^{ec-e}$  which implies that  
			\begin{equation}
				\label{sceq}
				e+(c-2)e=ec-e<N.	
			\end{equation}
			Comparing (\ref{steq}) and (\ref{sceq}) yields that $l>(c-2)e\ge e$ since $c>2$.
			Therefore, $2l>l+e\ge N$.  
		\end{proof}
		
		Throughout,  by $\bar{f}^{(i)}$ we denote the image of the polynomial $f=\sum_{j=0}^{n}a_jx^j\in R[x]$, in $R/ M^i[x]$, that is  $\bar{f}^{(i)}=\sum_{j=0}^{n}\bar{a_j}^{(i)}x^j$, where $\bar{a_j}^{(i)}$ is the image of 
		$a_j$ in $R/M^i$.
		
		\begin{lemma}\label{dervcond}
			Let $R$ be a finite non-commutative local ring  of $Char(R)=p^c$ with $c>1$, and let $f\in R[x]$ be a permutation polynomial on $R$. Then    
			$\bar{f}^{(i)}$ is a permutation polynomial on $R/M^i$ for $i=1,\ldots,N$.
			
		\end{lemma}
		\begin{proof}
			Fix $1\le i\le N$ and let $\bar{b}\in R/M^i$. Then there exists $b\in R$ such that $b=\bar{b} \mod M^i$.
			Since $f$ is a permutation polynomial on $R$ there exists $a\in R$ such that $f(a)=b$.
			But  then we have that  $\bar{f}^{(i)} (\bar{a})=\bar{b} \mod M^i$, that is $\bar{f}^{(i)}$ is surjective on $R/M^i$.
		\end{proof}
		By definition,  a finite local ring is a $p$-group with respect to addition. Hence,   the order of each element of $R$  is a power of $p$ that divides $Char(R)$. When $R$ is a finite chain ring,  De Luis~\cite{ordelchain} obtained a formula for computing the  order of the elements of $R$.
		\begin{lemma}\cite[Proposition 2]{ordelchain}\label{orderel}
			Let $R$ be a finite chain ring of  $Char(R)=p^c$ ($c>1$),   and maximal ideal $M$  of  nilpotency index $N$. If $a\in M^i\setminus M^{i+1}$ for some $0\le i\le N-1$, then the order of 
			$a$ is $p^{\lceil\frac{N-i}{e}\rceil}$, where $e$ is the ramification index of $R$.
			
		\end{lemma} 
		\begin{proposition}\label{forremovelam}
			Let $R$ be a finite non-commutative chain ring  of $Char(R)=p^c$ with $c>1$, and let $f\in R[x]$ be a permutation polynomial on $R$. Then   the following statements hold:
			\begin{enumerate}
				\item $f'(a)\ne 0 \mod M$ for every $a\in R$;
				\item $[\lambda_{ f}(y,z)]$ is a local permutation in $z$.
			\end{enumerate}
		\end{proposition}
		\begin{proof}
			(1)   Assume to the contrary that $f'(a)= 0 \mod M$ for some $a\in R$.  Since $c>1$, $0\ne p^{c-1}\in C(R)$  with $p^{2({c-1})}=0$.  Hence, by Lemma~\ref{calfor2z},
			\begin{equation}\label{fornoninj}
				f(a+p^{c-1})=f(a) +f'(a)p^{c-1}.
			\end{equation} 
			Now, if $f'(a)p^{c-1}=0$ then $f(a+p^{c-1})=f(a)$ which contradicts the fact that $f$ is a permutation
			polynomial on $R$. So, we may assume that $f'(a)p^{c-1}=p^{c-1}f'(a)\ne 0$, thus the order of $f'(a)$ with respect to  addition is $p^c$. Apparently,   the element $p^{c-1}=p^{c-1}.1_R$ has order $p$ with respect to addition. Hence, if $1\le i_1,i_2<N$ such that $f'(a)\in M^{i_1}\setminus M^{i_1+1}$  and  $p^{c-1}\in M^{i_2}\setminus M^{i_2+1}$, then by Lemma~\ref{orderel}, $c=\lceil \frac{N-i_1}{e}\rceil$ and  $1=\lceil \frac{N-i_2}{e}\rceil$. Thus, since    the least integer function is increasing  and  $c>1$, $\frac{N-i_1}{e}> \frac{N-i_2}{e}$, whence $i_1<i_2$.
			Therefore, Equation ~(\ref{fornoninj}) becomes \[f(a+p^{c-1})=f(a) \mod M^{i_2+1}. \] 
			But $p^{c-1}\ne 0 \mod M^{i_2+1}$, and whence $\bar{f}^{(i_2+1)}$ is not a permutation polynomial on 
			$R/M^{i_2+1} $ which contradicts  Lemma~\ref{dervcond}. Therefore,  there exists no $a\in R$ such that 
			$f'(a)=0 \mod M$.
			
			(2) Assume to the contrary that there exists $a\in R$ such that $[\lambda_{ f}(a,z)]$ is not a permutation of $R$. So, there exist $b_1,b_2\in R$ such that
			$\lambda_{ f}(a,b_1)=\lambda_{ f}(a,b_2)$ and $b_1\ne b_2$.
			Thus, by Fact~\ref{lamdapro}
			$$\lambda_{ f}(a,b )=0, \text{ where }   b=b_1- b_2\ne 0.$$ Now, either $pb\ne0$ or $pb=0$.
			First, consider the case where $pb\ne0$, and let   $l$ be the largest integer such that $p^lb\ne 0$. Then, we can apply Lemma~\ref{togennecha} by taking $m=p^lb$ together with the fact that $p^l\in C(R)$, to obtain 
			\[f(a+p^lb)=f(a)+\lambda_{ f}(a,p^lb)=f(a)+p^l\lambda_{ f}(a,b)=f(a),\] which contradicts  the fact that $f$ is a permutation polynomial on $R$.
			
			Second, consider the case  $pb=0$. Here, we discuss two cases on the number $c>1$.
			We begin with the case $c>2$. Then $b^2=0$ by Lemma~\ref{casec>2}. So,   by applying Lemma~\ref{togennecha}, we see that
			\[f(a+b)=f(a)+\lambda_{ f}(a,b)=f(a).\] Again we reach a contradiction.
			Finally,  assume that $c=2$. Without loss of generality, we can  suppose that $b^2 \ne 0$, since otherwise we can argue as in the
			case $c>2$. So, $b\in  M^s\setminus M^{s+1}$ for some $1\le s <N-1$. Thus, $b^2= 0 \mod M^{s+1}$.
			Therefore, we can apply  Lemma~\ref{togennecha} in the chain ring $R/ M^{s+1}$ to find that,
			\[f(a+b)=f(a)+\lambda_{ f}(a,b)  \mod  M^{s+1}=f(a) \mod  M^{s+1}.\]
			But  $b\ne 0 \mod M^{s+1}$, and thus $\bar{f}^{(s+1)}$ is not 	a permutation polynomial on $R/ M^{s+1}$ which again conflicts with Lemma~\ref{dervcond}. So, the assumption $[\lambda_{ f}(a,z)]$
			is not a permutation of $R$ for some $a\in R$ leads to a contradiction. Therefore, there is no such $a$ and $[\lambda_{ f}(y,z)]$ is a local permutation in $z$.
		\end{proof}
		We are now able to remove the condition on the assigned polynomial  in Theorem~\ref{Cherper} for the class of finite rings with $Char(R)\ne p$.
		\begin{theorem}\label{remlamd} 
			Let  $R$ be a finite chain ring of characteristic $p^c$ ($c>1$). Let $f =f_0+ \sum\limits_{i=1}^{k}f_i \alfa_i$, 
			where $f_0,\ldots,f_k \in R[x]$. Then 
			the following statements are equivalent:
			
			\begin{enumerate}
				\item 	$f$ is a permutation polynomial
				on $R_k$;
				\item 	$f_0$ is a permutation polynomial
				on $R_k$;
				
				\item $f_0$ is a permutation polynomial on $R$.
			\end{enumerate}
		\end{theorem}
		\begin{proof}
			We have already seen in Theorem~\ref{Cherper} that (1) and (2) are equivalent.
			Also, by Lemma~\ref{perpurcof}, (2) implies (3). Now, suppose $f_0$ is a permutation polynomial
			on $R$. Then, by Proposition~\ref{forremovelam}, $[\lambda_{ f_0}(y,z)]$ is a local permutation in $z$. Thus, this together with the fact that $f_0$ is a permutation polynomial on $R$ implies that $f_0$ is a permutation polynomial on $R_k$. 
		\end{proof}
		\begin{corollary}\label{nullandlocal}
			Let $R$ be a finite chain ring of characteristic $p^c$ ($c>1$) and let $f\in R[x]$ be a permutation polynomial on $R$. Then
			$[\lambda_{f+g} (y,z)] $ is a local  permutation    in the variable $z$ for every $g\in \Null$.
		\end{corollary}
		\begin{proof}
			Let $f\in R[x]$ be a permutation polynomial on $R$ and let $g\in \Null$. Then $f+g$ is a permutation polynomial on $R$ since $[f+g]_R=[f]_R$. Thus, by  Theorem~\ref{remlamd}, $f+g$ is a permutation polynomial on $R_k$. Therefore, $[\lambda_{f+g} (y,z)] $ is a local  permutation   in the variable $z$ by Theorem~\ref{Cherper}.
		\end{proof}
		\begin{corollary}\label{nulzerosum}
			Let $R$ be a finite chain ring of characteristic $p^c$ ($c>1$), $a\in R$ and   $g \in \Null$. Then
			\begin{enumerate}
				\item $ \sum\limits_{b\in R} \lambda_{g} (a,b)=0$;
				\item $ \sum\limits_{b,c\in R}\lambda_{g} (c,b)=0$.
			\end{enumerate}
		\end{corollary} 
		\begin{proof}
			Fix an element $a\in R$ and let $g\in \Null$ be arbitrary.
			Then	applying Corollary~\ref{nullandlocal} to the  permutation polynomial $f=x$   implies that  $[\lambda_{f} (a,z)] $ and $[\lambda_{f+g} (a,z)] $ are    permutations on $R$. Hence,
			\[R=\{\lambda_{f+g} (a,b)\mid b\in R \}=\{\lambda_{f} (a,b)\mid b\in R \}.\]
			Therefore, \[\sum\limits_{b\in R} b=  \sum\limits_{b\in R}\lambda_{f+g} (a,b)=\sum\limits_{b\in R}\lambda_{f} (a,b),\]
			whence by Fact~\ref{lamdapro}, \[  \sum\limits_{b\in R}\lambda_{f} (a,b)+\sum\limits_{b\in R}\lambda_{g} (a,b)=\sum\limits_{b\in R}\lambda_{f} (a,b).\]
			Thus, \[\sum\limits_{b\in R} \lambda_{g} (a,b)=0.\]
			This proves (1). Now (2) follows from (1).
		\end{proof}
		In general we have the following result.
		\begin{proposition}
			Let $R$ be a finite non-commutative ring, $a\in R$ and $g\in \Null$. Suppose that there exists a  polynomial   $f\in R[x]$  such that $f$ and $f+g$ are permutation polynomials on $R_k$. Then \begin{enumerate}
				\item $ \sum\limits_{b\in R} \lambda_{g} (a,b)=0$;
				\item $ \sum\limits_{b,c\in R}\lambda_{g} (c,b)=0$.
			\end{enumerate}
		\end{proposition}
		\begin{proof}
			Fix $a\in R$ and fix $g\in \Null$. Suppose the existence of a polynomial $f\in R[x]$ such that $f$ and $f+g$ are permutation polynomials on $R_k$. Hence, by Theorem~\ref{Cherper},  $[\lambda_{f} (a,z)] $ and $[\lambda_{f+g} (a,z)] $ are    permutations on $R$. Then the same argument used in the proof of Corollary~\ref{nulzerosum} works here as well.	
		\end{proof}
		\section{The group of pure polynomial permutations}\label{6}
		In this section, we consider a subgroup of $ \overline{\PrPol[R_k]}$ whose elements  are generated by polynomials over  $R$.  We call this group the group of pure polynomial permutations. Also, for the commutative case, we show this group is the complement of the group $\mathcal{ P}_{x\,,k}$ (see Proposition~\ref{nicegroups} and Remark~\ref{hascomp}) in the group of polynomial permutations. This  will allow us to pose a similar question for the non-commutative case.  
		Finally, in a subsection, we consider a subgroup
		of $\overline{\PrPol[R_k]}$ which  maps every element of $R$ to itself.
		
		Recall that from Corollary~\ref{clospolyf} that  for a finite non-commutative ring $R$ the closure $ \overline{\PrPol[R_k]}$ is a group. Therefore, in view of Fact~\ref{closis},  the closure $ \overline{B}$ is a subgroup of  $ \overline{\PrPol[R_k]}$ for every non-empty  subset $B\subseteq   \overline{\PrPol[R_k]}$. We use this fact implicitly in the rest of the paper. 
		\begin{Notation}\label{puregroup}
			
			For $j\ge1$ let	$$
			\mathcal{P}_R(R_j)=\{F\in \PrPol[R_j]\mid  F=[f]_{R_j} \textnormal{ for some }
			f\in R[x]\}
			.$$
		\end{Notation}
		Let $F_1,F_2\in \mathcal{P}_R(R_k)$ such that $F_1\ne F_2$. Then, by definition, $F_1=[f]_{R_k}$ and $F_2=[g]_{R_k}$ for some $f,g\in R[x]$. Furthermore, by Corollary~\ref{CHSGencount}, $[f]_R\ne [g]_{R}$ or $[\lambda_{ f}(y,z)]\ne [\lambda_{ g}(y,z)]$ since $F_1\ne F_2$. Thus,  
		one can see that the number $L$ defined in Proposition~\ref{CHSGeperncount} is the cardinality of the set $\mathcal{P}_R(R_k)$, and therefore, because of Theorem~\ref{Cherper}, the following result  holds.
		
		\begin{corollary}\label{Ldiscrip}
			Let $R$ be a finite non-commutative  Ring. Then
			\begin{align*}
				\left|\mathcal{P}_R(R_k)\right| &=\left|\{([f]_R,[\lambda_f(y,z)])\mid 
				f\in R[x], [f]_R\in \PrPol \text{ and } 
				[\lambda_f(y,z)] \text{ is a local permutation in } z\}\right|\\
				&	=\left|\{([f]_R,[\lambda_f(y,z)])\mid 
				f\in R[x], [f]_{R_k}\in \PrPol[R_k]\}\right|.
			\end{align*}
		\end{corollary}
		In the case $R$ is a finite chain ring of characteristic $p^c$ with $c>1$ by means of Theorem~\ref{remlamd}, we have the following
		\begin{corollary} Let $R$ be a finite chain ring with $Char(R)=p^c$ ($c>1$). Then
			\[	\left|\mathcal{P}_R(R_k)\right|  	=\left|\{([f]_R,[\lambda(y,z)])\mid 
			f\in R[x], [f]_{R}\in \PrPol[R]\}\right|.\]
		\end{corollary}
		In the following proposition, we show that the group  $ \overline{\mathcal{P}_R(R_j)}$ is a subgroup of $ \overline{\PrPol[R_j]}$, which is independent of the index $j$, and it is always embedded in $ \overline{\PrPol[R_k]}$ for every $k>j$.
		\begin{proposition}\label{CHSkthrelation}
			Let $R$ be a finite ring and let $k,j\ge 1$.	The group $ \overline{\mathcal{P}_R(R_k)}$ is a subgroup of $ \overline{\PrPol[R_k]}$, and
			$ \overline{\mathcal{P}_R(R_j)}\cong  \overline{\mathcal{P}_R(R_k)}$ for every $j\ne k$. In particular, $ \overline{\mathcal{P}_R(R_j)}$ is embedded in $ \overline{\PrPol[R_k]}$ for every $k>j$. 
		\end{proposition}
		\begin{proof}
			Since $\mathcal{P}_R(R_k)\subseteq \PrPol[R_k]$, we have $\overline{\mathcal{P}_R(R_k)}\subseteq \overline{\PrPol[R_k]}$ by Definition~\ref{closdef}. 
			 			
			Now, let $j\ne k$. Without loss of generality, assume that $k>j$ and  let $F\in \overline{\mathcal{P}_R(R_j)}$.  Then
			$F=F_1\circ F_2\circ \cdots \circ F_n$ for some $F_1,\ldots, F_n\in \mathcal{P}_R(R_j)$ by Definition~\ref{closdef}. Then, by the definition of   $ \mathcal{P}_R(R_j)$, $F_i=[f_i]_{R_j}$ for some $f_i \in R[x]$; $i=1,\ldots,n$.
			Define \[
			\psi\colon\overline{\mathcal{P}_R(R_j)} \longrightarrow \overline{\mathcal{P}_R(R_k)},\quad
			F\mapsto [f_1]_{R_k}\circ \cdots \circ  [f_n]_{R_k}.
			\]
			First, we show that $\psi$ is well defined. So,  assume that 
			there exist $l_1,\ldots,l_n\in R[x]$ such that $[l_i ]_{R_j}=F_i$ for $i=1,\ldots,n$. Then, by Corollary~\ref{CHSGencount}, $[l_i ]_{R_k}   =[f_i] _{R_k}$
			for $i=1,\ldots,n$. Thus, $$\psi(F)=[f_1]_{R_k}\circ\ldots \circ [f_n]_{R_k}=[l_1]_{R_k}\circ \ldots \circ  [l_n]_{R_k}.$$
			Also, if there exist $H_1,\ldots, H_m\in \mathcal{P}_R(R_j)$ such that $F=H_1\circ \ldots \circ H_m$. Then there exist $h_1,\ldots,h_m\in R[x]$ with $H_i=[h_i]_{R_j}$ for $i=1,\dots, m$. Thus, by Corollary~\ref{indepfromnumbe}, $[h_1]_{R_k}\circ\ldots [h_m]_{R_k}=[f_1]_{R_k}\circ\ldots\circ[f_n]_{R_k}=\psi(F) $.
			Therefore, $\psi$ is well defined.
			 Further,   it is  a one-to-one homomorphism by Corollary~\ref{indepfromnumbe}.
			  Now, let $G\in \overline{\mathcal{P}_R(R_k)}$. Then   $G=G_1\circ \cdots  \circ G_m$ for some $ G_1, \ldots, G_m\in \mathcal{P}_R(R_k)$, where  $G_i=[g_i]_{R_k}$  for some  
			$g_i\in R[x]$ for $i=1,\ldots,m$. But then $[g_i]_{R_j}\in  \PrPol[R_j]$ by Corollary~\ref{CHSPPfirstcoordinate}, whence $[g_i]_{R_j}\in  \mathcal{P}_R(R_j)$ (since $g_i\in R[x]$)  for $i=1,\ldots,m$. Thus,  $[g_1]_{R_j}\circ \cdots \circ [g_m]_{R_j}\in  \overline{\mathcal{P}_R(R_j)}$.  So, it is not hard to see that 
			\[\psi([g_1]_{R_j}\circ \cdots \circ [g_m]_{R_j})=[g_1]_{R_k}\circ \cdots \circ [g_m]_{R_k}.\]
			This shows that $\psi$ is surjective, whence it is an isomorphism. Therefore, 	$ \overline{\mathcal{P}_R(R_j)}\cong  \overline{\mathcal{P}_R(R_k)}$. Now, the last statement is obvious since $\overline{\mathcal{P}_R(R_k)}$ is a subgroup of $ \overline{\PrPol[R_k]}$.
		\end{proof}
		\begin{definition}
			We call the group $\overline{\mathcal{P}_R(R_k)}$ the group of pure polynomial permutations.
		\end{definition}
	
When $R$ is a finite commutative ring not only the set $\PrPol[R_k] $ is a group  but also its subset 
$\mathcal{P}_R(R_k)$, i.e., $\overline{\mathcal{P}_R(R_k)}= \mathcal{P}_R(R_k)$ (see~\cite{polysev}). In this case, we will show that   the group  $\mathcal{ P}_{x\,,k}$ defined in Proposition~\ref{nicegroups} is a normal subgroup
of $\PrPol[R_k] $  that admits a complement, namely  the group  $\mathcal{P}_R(R_k)$. Further, we will pose a related question for the non-commutative case involving  the group $\mathcal{ P}_{x\,,k}$ and the closure groups  $ \overline{\PrPol[R_k]}$,  $ \overline{\mathcal{P}_R(R_k)} $.
 		   
		Before doing so, we recall the definition of split extensions (see for example \cite[Page~760]{per1}). 
		The extension  $$1 \rightarrow H \xrightarrow{i} G \xrightarrow{p} N \rightarrow 1$$ is 
		called split if there is a homomorphism $l\colon N \rightarrow G$ with $p\circ l=id_N$. In such circumstances,
		the group $G$ is called the semi-direct product of $H$ by $N$, and we write  $G=H\rtimes N$.
			\begin{theorem}\label{complem}
			Let $R$ be a finite  commutative ring.	Let  $\PrPol[R_k] $ be the group of polynomial permutations on $R_k$ and let $\mathcal{P}_R(R_k)$ and
			$\mathcal{ P}_{x\,,k}$ as in Notation~\ref{puregroup} and Proposition~\ref{nicegroups} respectively. Then
			\begin{enumerate}
				\item $\PrPol[R_k]= \mathcal{ P}_{x\,,k} \rtimes \mathcal{P}_R(R_k)$;
				\item $|\PrPol[R_k]|= |\mathcal{P}_R(R_k)||\PolFun[R]|^k$.
			\end{enumerate}
		\end{theorem}
		\begin{proof}
			(1) Define a map $\varphi\colon \PrPol[R_k] \longrightarrow   \mathcal{P}_R(R_k)$ by
			$\varphi(F)=[f_0]_{R_k}$, where $F=[f]_{R_k}$ with  $f=f_0+\sum\limits_{i=1}^{k}f_i\alfa_i$ and $f_0,\ldots, f_k\in R[x]$. Then,  by the commutative form of  Corollary~\ref{CHSGencount} (see Remark~\ref{nulremark} and see~\cite[Corollary~3.6]{polysev}), $\varphi$ is well defined. Now, let $F_1\in  \PrPol[R_k]$ be induced by $g=g_0+\sum\limits_{i=1}^{k}g_i\alfa_i$, where $g_0,\ldots, g_k\in R[x]$. Then, since composition of polynomial functions and composition of polynomials are compatible over commutative rings,  we have  
			\[\varphi(F\circ F_1)=\varphi([f\circ g]_{R_k})= [f_0\circ g_0]_{R_k}\displaystyle_{=}[f_0]_{R_k}\circ  [g_0]_{R_k}=\varphi(F)\circ \varphi(F_1).\] 
			Thus, $\varphi$ is a homomorphism.  Then, it is evident that $\mathcal{ P}_{x\,,k} \subseteq\ker \varphi$.
			On the other hand if $F\in \ker \varphi$ is induced by $f=f_0+\sum\limits_{i=1}^{k}f_i\alfa_i$,
			then $\varphi(F)=[f_0]_{R_k}=id_{R_k}=[x]_{R_k}$.  Thus, by the commutative form of Corollary~\ref{CHSGencount}, $f_0\equiv x \mod \Nulld$, but this implies that  (again by the commutative form of Corollary~\ref{CHSGencount}) $F=[x+\sum\limits_{i=1}^{k}f_i\alfa_i]_{R_k}$. Thus, $F\in \mathcal{ P}_{x\,,k}$, and therefore  $ \mathcal{ P}_{x\,,k} =\ker \varphi$. 
			Now, by definition, $\varphi(F)=F$ for every $F\in \mathcal{P}_R(R_k)$, whence $\varphi$  is an epimorphism.
			Also, if $i_1\colon \mathcal{ P}_{x\,,k}\longrightarrow    \PrPol[R_k]$ and  $i_2\colon \mathcal{P}_R(R_k) \longrightarrow    \PrPol[R_k]$ are inclusion maps, then  $\varphi(i_2(F)) =F$ for every $F\in  \mathcal{P}_R(R_k)$ and the following extension is split 
			$$1 \rightarrow \mathcal{ P}_{x\,,k} \xrightarrow{i_1} \PrPol[R_k] \xrightarrow{\varphi} \mathcal{P}_R(R_k) \rightarrow 1.$$  
			Therefore, $\PrPol[R_k]= \mathcal{ P}_{x\,,k} \rtimes \mathcal{P}_R(R_k)$.
			(2) Follows from (1) and Proposition~\ref{nicegroups}.
			
		\end{proof}
		For the non-commutative case we have the following question.
		\begin{question}
			Let $R$ be a finite non-commutative ring. Then is it true that 
			\[ \overline{\PrPol[R_k]}= \mathcal{ P}_{x\,,k} \rtimes  \overline{\mathcal{P}_R(R_k)}?\]
		\end{question}
		\subsection{The set of stabilizer  polynomial permutations of $R$  }
		In this subsection, we consider a  subset of the set $ \PrPol[R_k]$ having the property of fixing (stabilizing) the elements of the ring $R$ pointwise.  We see that this  set has elements induced by pure polynomials (i.e., polynomials over $R$). 
		Also,   we 
		use the cardinality of this set to find  a counting formula for the number of polynomial permutations on $R_k$. Further, we consider  the closure group of this set.
		
		Similar to the commutative case (see~\cite[Definition~5.1]{polysev}), we have the following definition.
		\begin{definition}\label{CHSstd}
			Let
			$\Stabk=\{F\in \PrPol[R_k]\mid 
			F(a)=a \text{ for every } a\in R\}$. 
		\end{definition}
		Let $R$ be a finite ring.
		Then,	it is obvious that $\Stabk$ is a subset of $\PrPol[R_k]$, and therefore $\overline{\Stabk}$ is a subgroup of $\overline{\PrPol[R_k]}$.
		Next, we show that the elements of $\Stabk$ can be obtained by pure polynomials (see Definition~\ref{purpol}).   	
		\begin{proposition}\label{CHSfirststab}
			Let $R$ be a finite non-commutative ring. Then
			\[\Stabk=\{F\in \PrPol[R_k]\mid F 		\textnormal{ is induced by } x+h(x), h \in \Null[R] \}.\]		In particular, every element of $\Stabk$ is induced by a polynomial in $R[x]$.
		\end{proposition}
		\begin{proof}
			It is evident that 
			\[\Stabk[R]\supseteq\{F\in \PrPol[R_k]\mid F 
			\textnormal{ is induced by } x+h(x), h \in \Null[R] \}.\]
			For the converse, let $G\in \PrPol[R_k]$ 
			such that $G(a)=a$ for every $a\in R$. 
			Then $G$ is represented by $g_0+\sum\limits_{i=1}^{k}g_i\alfa_i$,  
			where $g_0,\ldots, g_k\in R[x]$, and  
			$a=G(a)=g_0(a)+\sum\limits_{i=1}^{k}g_i(a)\alfa_i$ for each $a\in R$. 
			It follows that $g_i(a)=0$ for each $a\in R$, i.e., 
			$g_i$ is a null polynomial on $R$ for $i=1,\dots,k$. 
			Thus  $g_0+\sum\limits_{i=1}^{k}g_i\alfa_i\quv g_0$ on $R_k$
			by Corollary~\ref{CHSGencount}, that is, $G$ is induced by $g_0$. 
			Further, $[g_0]_R= id_{R}$, where $id_{R}$ is the identity function on $R$, i.e., $g_0\equiv x \mod \Null$ and  therefore  $g_0(x)=x+h(x)$ for some $h\in \Null$. 
				\end{proof}
		
		\begin{definition}
			We call the group $\overline{\Stabk}$ the (pure) stabilizer closure group. 
		\end{definition}
		\begin{corollary}
			Let $R$ be a finite non-commutative ring. Then $\overline{\Stabk}$ is a subgroup of $ \overline{\mathcal{P}_R(R_k)}$.
		\end{corollary}
		Next, we show that for all $k\ge 1$ the stabilizer closure groups $\overline{\Stabk}$ are isomorphic.
		\begin{theorem}\label{CHSstabiso}
			Let	$k,j\ge 1$. Then
			$\overline{\Stabk}  \cong  \overline{ St_{ j}(R)}$. 
		\end{theorem}
		\begin{proof}
			Without loss of generality  assume that $k>j$
			 	and	let $F\in \overline{\mathcal{P}_R(R_j)}$.  Then
			$F=F_1\circ \cdots \circ F_n$ for some $F_1,\ldots, F_n\in \mathcal{P}_R(R_j)$, where $F_i=[f_i]_{R_j}$ and  $f_i \in R[x]$ for $i=1,\ldots,n$.
			Consider the isomorphism  \[
			\psi\colon\overline{\mathcal{P}_R(R_j)} \longrightarrow \overline{\mathcal{P}_R(R_k)},\quad
			F\mapsto [f_1]_{R_k}\circ \cdots \circ  [f_n]_{R_k} 
			\] defined in the proof of Proposition~\ref{CHSkthrelation}.
			Then $\overline{St_{ j}(R)}\cong \psi(\overline{St_{ j}(R)})$. Also, if $F\in \overline{St_{ j}(R)}$, then  $F=[x+h_1(x)]_{R_j}\circ \cdots \circ [x+h_n(x)]_{R_j}$ for some $h_1,\ldots,h_n \in\Null$  by Proposition~\ref{CHSfirststab}. Therefore, $\psi(F)=[x+h_1(x)]_{R_k}\circ \cdots \circ [x+h_n(x)]_{R_k}\in \overline{St_{ k}(R)}$
			since by Proposition~\ref{CHSfirststab}, $[x+h_i(x)]_{R_k} \in \Stabk$ for $i=1,\ldots,n$. This proves
			that $\psi(\overline{St_{j}(R)})\subseteq \overline{\Stabk[R]}$. Now, let $G\in \overline{\Stabk[R]}$. 
			Then, $G=[x+g_1(x)]_{R_k}\circ \cdots \circ [x+g_m(x)]_{R_k}$ for some $g_1,\ldots,g_m \in\Null$  by Proposition~\ref{CHSfirststab}. Again   by Proposition~\ref{CHSfirststab}, $[x+g_i(x)]_{R_j}\in St_{j}(R)$ for $i=1,\ldots,m$, and therefore 
			\[\psi([x+g_1(x)]_{R_j}\circ \cdots \circ [x+g_m(x)]_{R_j}) =[x+g_1(x)]_{R_k}\circ \cdots \circ [x+g_m(x)]_{R_k}=G\in \psi(\overline{St_{j}(R)})\] since $[x+g_1(x)]_{R_j}\circ \cdots \circ [x+g_m(x)]_{R_j}\in \overline{St_{ j}(R)}$. This proves the other inclusion and ends the proof.	
		\end{proof}
		By   definition, $\overline{ St_{ k}(R)}$ stabilizes $R$ pointwise, that is,  $F(a)=a$ for every $a\in R$ and for every $F\in \overline{ St_{ k}(R)}$. However, we do not know if 
		every element of $\overline{\PrPol[R_k]}$ that stabilizes the elements of $R$ pointwise is an element of  the stabilizer closure group. This motivates us to give the following definition.
		
		\begin{definition}\label{Rstb}
			Let
			$ Stb_k(R  )=\{F\in \overline{\mathcal{P}_R(R_k)}\mid 
			F(a)=a \text{ for every } a\in R\}$. 
		\end{definition}
		It is evident that $ Stb_k(R)$ is a subgroup of $\overline{\mathcal{P}_R(R_k)}$  whose elements  stabilize the elements of $R$ pointwise. Because of this, we call $ Stb_k(R)$ the stabilizer group of $R$ in  $\overline{\mathcal{P}_R(R_k)}$. Also, it will not be hard to see that the following inclusions of sets are true 
		\[\Stabk\subseteq \overline{\Stabk}\subseteq Stb_k(R). 
		\] An interesting question arising from the last relation is whether the following equalities hold:
		\[\Stabk= \overline{\Stabk}= Stb_k(R)?
		\]
		To answer this question affirmatively it would be enough to show that the set of polynomial permutations $\PrPol[R_k]$ is  closed with respect to composition. In particular, in the commutative case these equalities hold. 
		
		In the following, we obtain a general form of  \cite[Theorem~5.7]{polysev} for finite non-commutative rings. But we cannot ensure that the whole assertion of \cite[Theorem~5.7]{polysev} is valid in the non-commutative case. Because,
		in general, as we mentioned in Remark~\ref{isitredundant}-(\ref{qos}), we do not know if every permutation polynomial $f\in R[x]$ on the finite non-commutative ring $R$  is also 
		a permutation polynomial on the ring $R_k$ with the exception of a large class of chain rings (see Theorem~\ref{remlamd}).  
				\begin{proposition}\label{CHSPZlemma}
			Let $R$ be a finite ring. Then the stabilizer group $Stb_k(R) $ is a normal subgroup of the group 
			$\overline{\mathcal{P}_R(R_k)}$. Furthermore, if every element of $\PrPol$ is the restriction to $R$ of an element of $\mathcal{P}_R(R_k)$, then 
			\[
			\raise2pt\hbox{$\overline{\mathcal{P}_R(R_k)}$} \big/ \lower2pt\hbox{$Stb_k(R)$}
			\cong \overline{\mathcal{P}(R)}.
			\]
		\end{proposition}
		
		\begin{proof}
	
			Let  $F\in \overline{\mathcal{P}_R(R_k)}$. By definition, $F=  [f_1]_{R_k}\circ\cdots\circ[f_n]_{R_k}$  for some $f_1,\ldots,f_n\in R[x]$.  Now, define  
			a map	$\Psi\colon \overline{\mathcal{P}_R(R_k)} \longrightarrow \overline{\mathcal{P}(R)}$
			by $\Psi(F)= [f_1]_{R}\circ\cdots\circ[f_n]_{R}$. Then $\Psi$ is well defined by Theorem~\ref{Cherper} and Corollary~\ref{CHSGencount}. Evidently, it is 
			a group homomorphism with $\ker\Psi =  Stb_k(R)$. Therefore, by the First Isomorphism Theorem,\[
			\raise2pt\hbox{$\overline{\mathcal{P}_R(R_k)}$} \big/ \lower2pt\hbox{$Stb_k(R)$}
			\cong \Psi(\overline{\mathcal{P}_R(R_k)}).
			\]
			Furthermore, if the elements of   $\PrPol$ are obtained by  restricting elements of $\mathcal{P}_R(R_k)$ to $R$, then it is not  hard to see that $\Psi(\mathcal{P}_R(R_k))=\PrPol$. Hence,
			if $G\in \overline{\mathcal{P}(R)}$ such that $G=G_1\circ\cdots\circ G_m$ for some $G_1,\ldots,G_m\in\PrPol$, then there exist $F_1,\ldots, F_m\in \mathcal{P}_R(R_k)$ such that $\Psi(F_i)=G_i$ for $i=1,\ldots,m$. Therefore,
			\[
			G  =G_1\circ\cdots\circ G_m=\Psi(F_1)\circ\cdots\circ \Psi(F_m)=\Psi(F_1 \circ\cdots\circ  F_m)\in 
			\Psi(\overline{\mathcal{P}_R(R_k)}). \]  This shows that $	\Psi(\overline{\mathcal{P}_R(R_k)})=\overline{\mathcal{P}(R)}$.
		\end{proof}
		For chain rings of $Char(R) \ne p$, we can say more.
		\begin{theorem}\label{chainrestric}
			Let $R$ be a finite chain ring of $Char(R)=p^c$ with $c>1$. Then: 
			\begin{enumerate}
				\item each element of $\PrPol$ appears as a restriction on $R$ of some $G\in \mathcal{P}_R(R_k)$;
				\item   $Stb_k(R)$ is a normal subgroup of  
				$\overline{\mathcal{P}_R(R_k)}$ and 	\[
				\raise2pt\hbox{$\overline{\mathcal{P}_R(R_k)}$} \big/ \lower2pt\hbox{$Stb_k(R)$}
				\cong \overline{\mathcal{P}(R)}.
				\]
			\end{enumerate}
		\end{theorem}
		\begin{proof} 
			(1) Follows from Theorem~\ref{remlamd}.
			(2) This is a consequence of (1) and Proposition~\ref{CHSPZlemma}.
		\end{proof}
		\begin{remark}\label{imagepure}
			\leavevmode
			\begin{enumerate}
				\item Let $\Psi$ be the homomorphism in the proof of Proposition~\ref{CHSPZlemma}.  
				Then, since $\Psi^{-1}(id_R)=\ker \Psi= Stb_k(R)$,
				\[ |\ker \Psi\cap \mathcal{P}_R(R_k) |=|Stb_k(R)\cap \mathcal{P}_R(R_k)|=|\Stabk|.\]
				
				Therefore,
				\[L=|\mathcal{P}_R(R_k)|=|\Psi(\mathcal{P}_R(R_k))|\cdot|\Stabk|,\] where $L$ is as in  Proposition~\ref{CHSGeperncount} (see also the paragraph before Corollary~\ref{Ldiscrip}).
				
				\item When $R$ is commutative,  the homomorphism $\Psi$ in the proof of Proposition~\ref{CHSPZlemma}  is a surjection since in this case the restriction condition in the last statement of this proposition, is valid (see also \cite[Theorem~5.7]{polysev}), and in this case, therefore  $\Psi(\mathcal{P}_R(R_k))= \PrPol$.
				\item We noticed earlier that conditions on the formal derivatives in the commutative case were replaced by
				stronger conditions on the assigned polynomials, so, the reader could expect that the results in the commutative case involving  $\PrPol$ are still true in the non-commutative case by involving $\Psi(\mathcal{P}_R(R_k))$ instead. 
			\end{enumerate}
		\end{remark}
		From now on, let  $\Psi(\mathcal{P}_R(R_k))$ stand for the subset of $\mathcal{P}(R )$ obtained by restricting the elements of  $\mathcal{P}_R(R_k)$ to $R$.
		By   Remark~\ref{imagepure},  Proposition~\ref{CHSGeperncount} and Theorem~\ref{chainrestric}, we have the following.
		\begin{corollary}\label{CHS14}
			The number of polynomial permutations on $R_k$ is given by
			\[|\PrPol[R_k]|=|\PolFun[{R}]|^k\cdot
			|\Psi(\mathcal{P}_R(R_k))|\cdot |\Stabk|.\]
			In particular, when $R$ is a finite chain ring with $Char(R)=p^c$ ($c>1$),
			\[|\PrPol[R_k]|=|\PolFun[{R}]|^k\cdot
			|\PrPol |\cdot |\Stabk |.\] 	 
		\end{corollary}
		
		Similar to the commutative case, for every integer $n\ge 1$ we assign a subset to the ideal $\Null$ and a subset of the ideal $A\Null$. It is not difficult to see that these subsets are groups concerning the addition of polynomials. 
		\begin{definition} \label{CHS11.12}
			For $n\ge1$, we define
			\[\N[n]{R}=\{g\in R[x]\mid  g\in \Null[R]
			\textnormal{ with }\deg g < n\},\]  and
			\[		A\N[n]{R}=\{g\in R[x]\mid 
			g\in A\Null[R] \textnormal{ with }\deg g <n\}.\]
		\end{definition}

		In the following proposition, we give another description of the order of the stabilizer set $\Stabk[R]$. Also, we show this number is bounded by the index of $A\Null$ in $\Null$.
		\begin{proposition}\label{CHS12} 
			Let $R$ be a finite non-commutative ring. Then the following hold.
			\begin{enumerate}
				\item\label{CHSscndstab}
				$
				|\Stabk[R]|  =|\{[\lambda_g(y,z)]\mid  g\in \Null[R] \text{ and } [g+x]_{R_k}\in 
				\mathcal{P}_R(R_k)\}|.
				$
				\item\label{CHSthirdstab}
				If there exists a monic null polynomial on $R_k$ in $R[x]$ of degree~$n$, 
				then:
				\begin{enumerate}

					\item \label{CHSthirdstaba}
					
					$
					|\Stabk[R]|  =
					|\{[\lambda_g(y,z)]\mid  g\in \Null[R] \text{ and } [g+x]_{R_k}\in 
					\mathcal{P}_R(R_k) \textnormal{ with }\deg g<n\}|;
					$
					\item\label{CHSfourthstab} $
					|\Stabk[R]|  \le [\Null[R] \colon A\Null[R]]=
					\frac{|\N[n]{R}|}{|A\N[n]{R}|}.
					$
				\end{enumerate}
			\end{enumerate}
			
		\end{proposition}
		
		\begin{proof}
			(\ref{CHSscndstab})
			By Proposition~\ref{CHSfirststab},  
			\begin{align*}
				|\Stabk[R]| & =|\{[\lambda_{x+g}(y,z)]\mid  g\in \Null[R] \text{ and } [g+x]_{R_k}\in 
				\Stabk\}|\\
				& =|\{[\lambda_{g}(y,z)]\mid  g\in \Null[R] \text{ and } [g+x]_{R_k}\in 
				\mathcal{P}_R(R_k)\}|.
			\end{align*}
			Since for every $g\in \Null$, $[x+g]_{R_k}\in \Stabk$ if and only if $[x+g]_{R_k}\in \mathcal{P}_R(R_k)$.\\
			(2) By Remark~\ref{CHSexistmonicn}, there is always a  monic null polynomial  on  $R_k$  with coefficients from $R$.
			
			(\ref{CHSthirdstaba}) If $g\in \Null[R]$, then by 
			Proposition~\ref{CHSsur}, there exists 
			$f\in R[x]$ with $\deg f<n$ such that 
			$[\lambda_{f}(y,z)]=[\lambda_{g}(y,z)]$ and	$[f]_R=[g]_R$. Clearly, $f\in \Null[R]$.
			
			(\ref{CHSfourthstab})
			First, we show the inequality.
			We have, since every element of $\Stabk$ induces the identity on $R$,
			\[	\{[\lambda_{x+g}(y,z)]\mid  g\in \Null[R] \text{ and } [g+x]_{R_k}\in 
			\Stabk\} \subseteq \{[\lambda_{f}(y,z)]\mid  f\in R[x]\text{ and } [f]_{R}=id_R 
			\}.\]
			Thus, by (\ref{CHSscndstab}) and by Corollary~\ref{indecon} for the case $F=id_R$,
			\begin{equation}\label{stequal}	  
				|\Stabk|\le |  \{[\lambda_{f}(y,z)]\mid  f\in R[x]\text{ and } [f]_{R}=id_R 
				\}|=[\Null \colon A\Null].\end{equation}
			To compute the ratio,  set $\mathcal{B}=\{[\lambda_{f}(y,z)]\mid  f\in R[x]  
			\}$. Then, by Fact~\ref{lamdapro},  $\mathcal{B}$ is an additive group with operation defined as the following \[[\lambda_{f}(y,z)]+[\lambda_{g}(y,z)]=[\lambda_{f+g}(y,z)].\]
			We leave it to the reader to check the details.
			
			Now	define $\phi\colon \Null[R] \longrightarrow \mathcal{B}$
			by $\phi(f)=[\lambda_{f}(y,z)]$.
			
			Then, it is obvious that $\phi$ is a homomorphism of additive groups. 
			By the definition of $A\Null$, $  \ker\phi=A\Null$, and thus  $\raise1.5pt\hbox{$\Null[R]$} \big/ \lower1.5pt\hbox{$A\Null$} \cong \im( \phi)$.
			
			So, if $\phi_1$ stands to the restriction of $\phi$ to the subgroup $\N[n]{R}$, we have  similarly that
			$  \ker\phi_1=A\N[n]{R}$, and    $\raise1.5pt\hbox{$\N[n]{R}$} \big/ \lower1.5pt\hbox{$A\N[n]{R}$} \cong \im (\phi_1)$.
			Therefore, $\frac{|\N[n]{R}|}{|A\N[n]{R}|}=|\im( \phi_1)|$.  So to end the proof, we need only show that $\im (\phi)\subseteq\im (\phi_1)$ since the other inclusion is valid.
			Assume that $F\in \im (\phi)$. Then there exists $g\in \Null$ such that $\phi(g)= [\lambda_{g}(y,z)]=F$.
			By part (a), there exists $f\in N_R$ with $\deg f<n$ (i.e., $f\in\N[n]{R}$) such that $[\lambda_{g}(y,z)] =[\lambda_{f}(y,z)]$ but this means that $F=[\lambda_{f}(y,z)]=\phi_1(f)\in \im (\phi_1)$.
		\end{proof}
		Again on the considered class of finite chain rings, one can say more. For instance, the inequality of Proposition~\ref{CHS12} becomes indeed equality. The following theorem illustrates this.
		\begin{theorem}
			Let $R$ be a finite chain ring with $Char(R)=p^c$ ($c>1$).	Then the following hold.
			\begin{enumerate}
				\item\label{CHaSscndstab}
				$
				|\Stabk[R]|  =|\{[\lambda_g(y,z)]\mid  g\in \Null[R] \}|.
				$
				\item\label{CHaSthirdstab}
				If there exists a monic null polynomial on $R_k$ in $R[x]$ of degree~$n$, 
				then:
				\begin{enumerate}

					\item \label{CHaSthirdstaba}
					
					$
					|\Stabk[R]|  =
					|\{[\lambda_g(y,z)]\mid  g\in \Null[R] \text{ and }  \deg g<n\}|;
					$
					\item\label{CHaSfourthstab} $
					|\Stabk[R]|  = [\Null[R] \colon A\Null[R]]=
					\frac{|\N[n]{R}|}{|A\N[n]{R}|}.
					$
				\end{enumerate}
			\end{enumerate}
		\end{theorem}
		\begin{proof}
			(\ref{CHaSscndstab}) Follows by Proposition~\ref{CHS12}-(\ref{CHSscndstab}) since $[f+x]_{R_k}\in \PrPol[R_k]$  for every $f\in \Null$ by Theorem~\ref{remlamd}.
			
			(\ref{CHaSthirdstaba}) Follows from (\ref{CHaSscndstab}) and Proposition~\ref{CHSsur}.
			
			(\ref{CHaSfourthstab})	By Proposition~\ref{CHS12}, we need only show that 	$|\Stabk[R]|  = [\Null[R] \colon A\Null[R]]$. Now, let $f\in R[x]$. Then, by Theorem~\ref{remlamd}, $[f]_{R_k} \in \Stabk[R]$ if and only if $[f]_R=id_R$. Therefore, by Corollary~\ref{CHSGencount}, 
			\[
			|\Stabk|= |  \{[\lambda_{f}(y,z)]\mid  f\in R[x]\text{ and } [f]_{R}=id_R. 
			\}|\] But then comparing this relation with Equation~(\ref{stequal})  in the proof of Proposition~\ref{CHS12} ends the proof. 
		\end{proof}
		The previous theorem, Corollary~\ref{CHS14} and Proposition~\ref{CHSfirstcountfor} imply the following result.
		
		\begin{corollary}
			Let $R$ be a finite chain ring with $Char(R)=p^c$ ($c>1$). Then
		
		\begin{enumerate}
			\item $|\PolFun[R_k]|=|\Stabk||\PolFun|^{k+1}$; 
			\item $\frac{|\PrPol[R_k]|}{|\PolFun[R_k]|}=\frac{|\PrPol|}{|\PolFun|}$.
		\end{enumerate}	
		\end{corollary}
		
	\noindent {\bf Acknowledgment.} 
	This research was funded in part  by the Austrian Science Fund (FWF) [10.55776/P35788]. For open access purposes, the authors have applied a CC BY public copyright license to any
	author-accepted manuscript version arising from this submission.
	 The second author is partially supported by the IMU-Simons African Fellowship Program. The second author would like to thank the Department of Mathematics at the University of Graz for kind hospitality during her stay in Graz. The authors also thank the referee for the careful reading of the manuscript and for the valuable comments that helped to improve the manuscript. 

	\bibliographystyle{plain}
	\bibliography{noncomdu}
\end{document}